\documentclass[,]{informs3}

\OneAndAHalfSpacedXII 

\usepackage{endnotes}
\let\footnote=\endnote
\let\enotesize=\normalsize
\def\notesname{Endnotes}%
\def\makeenmark{\hbox to1.275em{\theenmark.\enskip\hss}}
\def\enoteformat{\rightskip0pt\leftskip0pt\parindent=1.275em
  \leavevmode\llap{\makeenmark}}

\usepackage{algorithm}
\usepackage{algpseudocode}
\usepackage{tikz}
\usepackage[utf8]{inputenc}
\usepackage{amsmath,amssymb,amsfonts}
\usepackage{mathtools}
\usepackage{hyperref}
\usepackage{caption,subcaption}
\usepackage{graphicx}
\usepackage{xcolor}
\usepackage{pgfplots}
\usetikzlibrary {arrows.meta}
\usepackage[dvipsnames]{xcolor}

\newcommand{\brho}{\rho}

\usepackage{natbib}
 \bibpunct[, ]{(}{)}{,}{a}{}{,}%
 \def\bibfont{\small}%
\TheoremsNumberedThrough     
\ECRepeatTheorems

\EquationsNumberedThrough    

\renewcommand{\theARTICLETOP}{}
\begin{document}


\RUNAUTHOR{Author}

\RUNTITLE{Dynamic Traffic Allocation for Revenue Maximization on Creator Economy Platforms}

\TITLE{Dynamic Traffic Allocation for Revenue Maximization on Creator Economy Platforms}

\ARTICLEAUTHORS{%
\AUTHOR{Zhengli Wang}
\AFF{Faculty of Business and Economics, the University of Hong Kong, \EMAIL{wzl1@hku.hk}}

\AUTHOR{Lin (Franklin) Feng}
\AFF{Stanford Graduate School of Business, \EMAIL{linfeng1@stanford.edu}}

\AUTHOR{Zhixi Wan}
\AFF{Faculty of Business and Economics, the University of Hong Kong,  \EMAIL{zhixiwan@hku.hk}}
} 

\ABSTRACT{%
Creator economy platforms face a strategic dilemma: allocating traffic to established stars for immediate ad revenue versus nurturing emerging creators to build a follower base for future monetization. We develop a continuous-time dynamic optimization model to characterize the optimal traffic allocation policy for a platform managing heterogeneous creators with dual revenue streams (advertising and direct follower contributions). We characterize the optimal policy analytically, revealing a ``most-valuable-creator-first" rule driven by a forward-looking activation set. Under Bass diffusion dynamics, this policy exhibits a sophisticated ``conditional reversal" strategy, where the platform temporarily prioritizes lagging creators to capitalize on word-of-mouth effects. Regarding the ecosystem structure, we find the optimal policy acts as a selective gatekeeper. Unlike myopic policies that lead to a harsh ``winner-take-all" market, or naive fairness-driven heuristics that foster inefficient ``indiscriminate growth," the optimal policy imposes a strict capability threshold that screens which creators receive platform traffic. Furthermore, it enforces disciplined growth for successful entrants, capping their follower bases at an optimal ceiling to prevent over-investment. Finally, we demonstrate that simple heuristics can lead to significant revenue losses (up to 25\%) and propose a practical ``follower-growth adjusted" heuristic that achieves near-optimal performance by leveraging observed growth momentum.
}%



\maketitle

%


\section{Introduction}

The creator economy, a vibrant digital ecosystem where over 50 million individuals create content for platforms like YouTube, TikTok, and Instagram, has grown into a significant economic force \citep{forbes_2022_creator_economy}. The scale of this activity is immense, with over a billion hours of video consumed daily on YouTube alone and the sector valued at an estimated \$13.8 billion in 2021 \citep{forbes_creator_valuation_2021}. 

At the heart of this ecosystem lies content monetization, which serves as both the primary revenue source for platforms and the key incentive for professional creators. Historically, online advertising revenue, pioneered by platforms like YouTube with its Participatory Video Ads introduced in 2006, has been the dominant monetization mode \citep{forbes_2006_youtube_pva}. In this model, the platforms act as the central intermediary connecting viewers, creators, and advertisers \citep{bhargava2022creator}. Their recommendation algorithms influence content exposure, often disproportionately allocating viewer traffic to a small number of ``star creators", creating a highly skewed ecosystem where the vast majority struggle to earn a living wage. For instance, achieving just \$2000 from YouTube's ad revenue share reportedly requires 1 million views \citep{economist_creator_rules_2021}, making it challenging to cultivate a robust ``middle class" of creators. Because ad revenue requires massive viewership scale to generate a livable income, the traditional model inherently starves emerging niche creators of financial viability.
 
In response to these challenges, new monetization models, collectively termed follower-based monetization, have rapidly emerged. These modes include subscriptions for premium content, direct tips or donations from ``true fans," and influencer marketing deals with brands \citep{forbes_2022_creator_economy}. Unlike ad-revenue, follower-based monetization fundamentally alters platform economics: it introduces the possibility of a creator ``middle class'' capable of sustaining themselves on a much smaller base of highly loyal followers. Platforms are now actively facilitating these new revenue streams, creating a complex dual-revenue environment where traffic allocation has both direct and indirect financial consequences \citep{youtube_vidcon_2018_memberships}.

Despite the growing prominence of this hybrid model, there has been limited academic research exploring its implications for platform strategy. The existing literature primarily focuses on ad-revenue sharing or general content promotion \citep{bhargava2022creator,jain2021compensating}. This leaves a critical gap in understanding how platforms should manage this modern dual-revenue environment. To address this, our paper investigates the following three research questions:

\textbf{Research Question 1:} What is the structure of the optimal dynamic traffic allocation policy for a platform seeking to maximize total revenue from both advertising and follower-based channels? This question addresses the central operational dilemma for a platform manager. This is not merely a theoretical puzzle but a core strategic challenge that the industry's largest players openly acknowledge. In practice, platforms utilize recommender systems—such as TikTok's For You feed—to deliberately distribute traffic to emerging creators \citep{tiktok_foryou_2020}, often prioritizing discovery over immediate monetization from established stars.

This strategic choice underscores the fundamental trade-off: should the algorithm direct a viewer to a proven creator for guaranteed immediate returns, or to an emerging one as an investment in future engagement? In response, we characterize the optimal policy, revealing a ``most-valuable-creator-first" rule where value is defined by a forward-looking activation set. Our analysis highlights dynamic patterns such as a ``conditional reversal" strategy, where it can be optimal to temporarily channel viewer traffic to a lagging creator to exploit accelerating word-of-mouth effects.

\textbf{Research Question 2:} How does the optimal policy shape the platform's long-term market structure? While industry observers hope that dual-revenue environments will naturally foster a diverse ``middle class'', does a platform's revenue-maximizing allocation actually achieve this, or does it reinforce a ``winner-take-all'' dynamic? We find that the optimal policy acts as a ``selective gatekeeper,'' striking a precise balance that simple heuristics miss. While myopic policies result in a harsh ``winner-take-all'' market and naive ``fairness'' heuristics (favoring smaller creators) lead to indiscriminate and inefficient growth, the optimal policy imposes a strict capability threshold. Furthermore, for those who qualify, the policy enforces ``disciplined growth,'' capping their follower bases at an optimal ceiling to prevent over-investment.

\textbf{Research Question 3:} Given the potential complexity of the optimal solution, how much value does it create over simple heuristics, and can a practical, data-driven heuristic be designed to achieve near-optimal performance? This inquiry leads to the search for a robust, data-driven heuristic that captures the core logic of the optimal policy. Our numerical analysis reveals that simple myopic policies can lead to significant revenue loss. Surprisingly, we find that the best data-driven heuristics are those that run counter to the common wisdom of ``betting on the winners." Further, we identify a ``follower-growth adjusted" heuristic that is not only managerially intuitive but also remarkably robust, achieving near-optimal results even with imperfect information.

To derive these findings, we develop a dynamic platform optimization model in continuous time. We model a platform with heterogeneous creators and analyze the problem under two common follower growth dynamics—a linear model and the Bass model. We first analytically characterize the optimal allocation policy, and then use extensive numerical experiments to evaluate its performance against practical heuristics.

To the best of our knowledge, this paper is the first in the operations management literature to analytically study the platform viewer traffic allocation problem within a dual-revenue environment. Methodologically, we contribute to the continuous-time optimal control literature by integrating active traffic allocation, heterogeneous creator capabilities, and non-concave Bass diffusion dynamics into a tractable analytical framework. Managerially, by explicitly modeling both short-run advertising income and long-run follower contributions, our structural results provide pioneering insights into how dynamic traffic allocation fundamentally shapes a platform's long-term market structure.

The remainder of this paper is structured as follows. Section 2 reviews the relevant literature. Section 3 introduces our model setup. Sections 4 and 5 characterize the optimal traffic allocation policies. Section 6 presents our numerical results. Finally, Section 7 concludes and discusses avenues for future research.

\section{Literature Review}

Compared with existing studies, our paper is distinctive in incorporating five features within a single platform optimization framework: (i) \emph{traffic allocation} as the platform’s lever; (ii) \emph{follower-based growth} where allocation has an impact for the future; (iii) \emph{two revenue streams}: short-run advertising and long-run subscription fees; (iv) \emph{heterogeneity} in content creator's capability and growth potential; and (v) an explicit characterization of the platform’s optimal policy structure, which is both interpretable and demonstrates the trade-off between immediate revenue and future growth.

In recent years, there has been growing interest in the academic literature to study online content platforms. \citet{caro2020managing} analyze how creators optimize posting intensity to build and monetize their follower base; \citet{jain2021compensating} study how advertising revenue sharing is shaped by competition among creators;  \citet{mai2023optimizing} investigate pricing and advertising strategies in an online game platform to maximize long-run profit; and \citet{li2026robust} investigates matching algorithms between creators and advertisers. Among this stream of literature, the two papers most relevant to us are \citet{lin2024content} and \citet{ma2025user}. \citet{lin2024content} develop a Promotion-Bass Diffusion Model that captures promotion and diffusion effects in online content adoption. They formulate a discrete-time optimization problem to maximize total adoptions. Different from their formulation, we use a continuous-time framework that accounts for both short-run advertising and long-run subscription fees, and we analytically characterize the explicit structure of the platform’s optimal policy. \citet{ma2025user} analyze an online content platform’s dynamic operational policies by modeling the triangular relationship among platform, creators, and viewers, assuming viewership traffic is shared equally among homogeneous creators. They adopt an optimal control approach to maximize advertising profits. While our study shares a similar context, our formulation is essentially different from theirs: our model allows the platform to allocate traffic across heterogeneous creators and incorporates subscription fees as part of the platform's revenue source.

Methodologically, our work builds on the broad literature of continuous-time stochastic control in Operations Management (OM), which has been applied across diverse domains such as queueing control \citep{lin2023wait}, platform-based resource allocation \citep{liang2023efficient}, healthcare service or clinical trial \citep{hu2025prediction,wang2020adaptive}, and adverse event monitoring \citep{chen2020optimal}. Related research has employed similar frameworks to dynamic pricing and inventory coordination \citep{feng2014dynamic}, capacity management in product launches \citep{shen2014optimal, sunar2019optimal}, the timing of strategic investment decisions \citep{kwon2010invest}, entrepreneurship \citep{wang2022new,wang2026strategies}, Bayesian learning \citep{harrison2015investment,sunar2024optimal}, role of production flexibility in capital management \citep{lai2023interplay}, effort allocation or incentive design \citep{dawande2019optimal,li2024leveling}, and dynamic growth in two-sided markets \citep{lian2021optimal}. Although these studies do not directly address problems in the creator economy, they share similar modeling techniques and incorporate elements closely related to ours, such as Bass-type diffusion processes, the trade-off between short-run and long-run objectives, and the characterization of optimal policy structures. For instance, \citet{hu2022managing} develop a continuous-time control model in which a manufacturer
of self-replicating innovative goods dynamically allocates a limited resource to maximize discounted profit. Their model employs a Bass-type diffusion process for demand and captures the tension between short-term sales revenue and long-term capacity growth, leading to a threshold-type optimal policy. Different from these studies, our model incorporates features such as traffic allocation, follower-based growth and creator heterogeneity. These features are unique to creator economy context and yield insights that are particularly relevant to this emerging application context.  

\section{Model}\label{Sec:Model}

\subsubsection*{Creators and Revenue. }
We consider a content platform that hosts \(N\) heterogeneous creators, indexed by $i\in\mathcal{N}=\{1,\dots,N\}$, over a continuous-time, infinite horizon. Each creator's potential is defined by two unique parameters: (1) \textit{capability} $\brho_i>0$, that captures a creator's effectiveness at converting traffic into both advertising revenue and follower contributions. We assume creators are ordered by capability, such that $\brho_1>\brho_2>\dots>\brho_N$; (2) \textit{follower capacity} $m_i>0$, that represents the maximum potential size of creator $i$'s follower base. The heterogeneity in $m_i$ captures the varying total addressable markets of distinct content niches, which ranges from globally appealing mass entertainment to highly specialized topics.

The primary state variable for each creator $i$ is their number of followers at time $t$, denoted by $\mathcal{S}_i(t)$. The instantaneous revenue rate generated by creator $i$ is a function of the \textit{traffic} they receive, $\mathcal{A}_i(t)$, and their current follower base. It is composed of two parts: (1) \textit{per-exposure advertising revenue} $\brho_i\mathcal{A}_i(t)$ and (2) \textit{per-follower contribution revenue} (e.g., subscriptions or tips) $\brho_ir\mathcal{S}_i(t)$, where $r>0$ is the unit subscription fee. This structure highlights the dual role of a creator's capability $\rho_i$ in monetizing both immediate viewership and their accumulated base of loyal followers.

\subsubsection*{Control and Dynamics. }
The platform's primary control lever is the dynamic allocation of its total viewer traffic, $A>0$. At each instant $t$, the platform decides an allocation $\mathcal{A}(t)=(\mathcal{A}_1(t),\dots,\mathcal{A}_N(t))$, where $\mathcal{A}_i(t)\geq 0$ denotes the traffic assigned to creator $i$, subject to the capacity constraint $\sum_{i=1}^N\mathcal{A}_i(t)=A$.

Traffic allocation is critical because it directly determines the growth trajectory of each creator’s follower base. We model follower dynamics as proportional to the allocated traffic, with evolution governed by the ordinary differential equation (ODE):
\begin{align}
\frac{d\mathcal{S}_i(t)}{dt} = \mathcal{A}_i(t) \mathcal{F}(\mathcal{S}_i(t); \brho_i, m_i),\label{eq:dynamics}
\end{align}
where $\mathcal{F}(\cdot; \brho, m)$ is a general function that represents the evolution of follower growth. We assume that follower growth occurs independently across creators.

\subsubsection*{Platform's Optimization Problem.}
The platform's objective is to determine a dynamic traffic allocation policy that maximizes the total discounted revenue generated by all creators. A policy, $\gamma$, is a function that maps any given state of follower bases, $S(t)$, to a feasible traffic allocation
\begin{align*}
    \gamma(\mathcal{S}(t))\in\left\{(\mathcal{A}_1,\dots,\mathcal{A}_N): \sum_{i=1}^N\mathcal{A}_i=A,\mathcal{A}_i\geq 0\right\}.
\end{align*}

Formally, the platform seeks to find the optimal policy $\gamma^*$ from the set of all admissible policies $\Gamma$ to solve the following continuous-time optimal control problem:
\begin{align}
    \max_{\gamma\in\Gamma}&\; \int_0^{+\infty}e^{-\delta t}\sum_{i=1}^N \brho_i[\mathcal{A}_i(t) + r\mathcal{S}_i(t)]\,dt :=\max_{\gamma\in\Gamma}\Pi \label{eq:GenObj}
\end{align}
subject to the follower growth dynamics \eqref{eq:dynamics} for each creator $i\in\mathcal{N}$ and the initial conditions $\mathcal{S}_i(0)=S_i$. Here, $\delta$ is the discount factor, and $\mathcal{A}_i(t)=\gamma_i(\mathcal{S}(t))$.

\subsubsection*{Follower Growth Models: Linear and Bass.}
To make the problem tractable and derive analytical insights, we specify the general follower growth function $\mathcal{F}(\cdot;\rho,m)$ using two common functional forms from the literature.
\begin{itemize}
    \item \textit{Linear Model:} The first model assumes that the follower conversion rate decreases linearly as a creator's follower base approaches its capacity. The growth function is
    \begin{align}
        \mathcal{F}_L(\mathcal{S}_i(t);\brho_i,m_i)=p\brho_i\left(1-\frac{\mathcal{S}_i(t)}{m_i}\right).\label{Eq:Ln_Evolution}
    \end{align}
    Here, $p>0$ is an exogenous conversion rate. This model represents that the growth rate is proportional to the ``remaining'' follower-base percentage (i.e, $1-\frac{\mathcal{S}_i(t)}{m_i}$) and, thus, the growth is fastest at the beginning.
    
    \item \textit{Bass Model:} The second model incorporates a word-of-mouth effect, where existing followers help attract new ones. The growth function is:
    \begin{align}
    \mathcal{F}_B(\mathcal{S}_i(t); \brho_i, m_i) :=  \left (p\brho_i + q \cdot \frac{\mathcal{S}_i(t)}{m_i} \right) \left( 1 - \frac{\mathcal{S}_i(t)}{m_i} \right).\label{Eq:Bs_Evolution}
    \end{align}
    This formulation follows the Bass diffusion model \citep{bass1969new}, which has since been widely been adopted in operations management \citep{shen2014optimal, ma2025user, hu2022managing, lin2024content}. In addition to the direct conversion rate $p$, the model incorporates an internal influence term, with $q>0$ capturing the strength of word-of-mouth effects. This yields a non-concave, S-shaped growth trajectory: follower growth accelerates during intermediate stages before tapering off as saturation approaches. Note that the linear model is a special case of the Bass model, obtained by setting $q=0$.
\end{itemize}

\section{Optimal Policy: Two Creators}\label{Sec:Opt_policy_2Crtr}
In this section, we characterize the optimal traffic allocation policy in a simplified setting with two heterogeneous creators. The analysis reveals a powerful and intuitive structure: the platform follows a ``most valuable creator first" rule, allocating all traffic to the creator who offers the greatest marginal return at each moment. Central to this policy is a forward-looking \textit{marginal value index}, which accounts for both the immediate advertising revenue from an extra unit of traffic and the discounted future value of the followers it generates. 

We first derive this policy under the linear growth model (Section~\ref{Subsec:Opt_policy_2Crtr_Ln}), and then extend the analysis to the Bass model where word-of-mouth effects enrich the dynamics (Section~\ref{Subsec:Opt_policy_2Crtr_Bs}). Finally, we examine Research Question 2 by characterizing the conditions under which the creators can receive positive traffic, as well as their long-run follower bases (Section~\ref{Subsec:2Crtr_Long_Base}).

\subsection{The Linear Model}\label{Subsec:Opt_policy_2Crtr_Ln}
In Section~\ref{Subsubsec:Opt_policy_2Crtr_Ln_op}, we derive the structure of the optimal allocation policy under the linear model by partitioning the state space and characterize the associated optimal control evolution. Section~\ref{Subsubsec:Opt_policy_2Crtr_Ln_it} provides an illustration of the optimal policy's structure and corresponding state evolution.

\subsubsection{The Optimal Policy.}\label{Subsubsec:Opt_policy_2Crtr_Ln_op}
The platform's objective is to maximize its total revenue $\Pi$ in \eqref{eq:GenObj}, subject to the linear growth dynamics. The resulting optimal policy is characterized in the following theorem.

\begin{theorem}\label{Thm:Opt_policy_2Crtr_Ln}
    Under the linear follower growth with $2$ creators, let the initial number of followers for creator $i=1$, $2$ be $S_1$, $S_2$ respectively. Then, there exist functions $g,h:\mathbb{R}\to\mathbb{R}$ such that the optimal traffic allocation policy $\gamma^*$ takes the form
    \begin{align}
        \gamma^*(\mathcal{S}_1,\mathcal{S}_2)&=(\mathcal{A}_1^*,\mathcal{A}_2^*)\nonumber\\
        &=
        \begin{cases}
            (A,0)&\text{if }\mathcal{S}_2>g(\mathcal{S}_1),\\
            (0,A)&\text{if }\mathcal{S}_2<g(\mathcal{S}_1),\\
            (h(\mathcal{S}_1),A-h(\mathcal{S}_1))&\text{if }\mathcal{S}_2=g(\mathcal{S}_1).
        \end{cases}\nonumber
    \end{align}
    In particular, the function $g$ is affine with the explicit form
    \begin{align}
        g(\mathcal{S}_1)=\frac{m_2}{m_1}\cdot\frac{\brho_1^2}{\brho_2^2}\mathcal{S}_1-\frac{m_2}{\brho_2^2}\left[\frac{\delta}{pr}(\brho_1-\brho_2)+(\brho_1^2-\brho_2^2)\right],\nonumber
    \end{align}
    and the function $h$ is given by
    \begin{align}
        h(\mathcal{S}_1)=A\cdot\frac{m_1\rho_2^3(m_2-g(\mathcal{S}_1))}{m_1\rho_2^3(m_2-g(\mathcal{S}_1))+m_2\rho_1^3(m_1-\mathcal{S}_1)}.\nonumber
    \end{align}
\end{theorem}

The economic intuition behind this affine partition, $\mathcal{S}_2=g(\mathcal{S}_1)$, is best understood through the marginal value index of each creator. This line represents an indifference curve where the marginal values of both creators are exactly equal. To formalize this, we define the time-dependent marginal value function for each creator $i$ as
\begin{equation}
\Psi_i(t) := \brho_i\left[1 + \frac{r}{\delta}\mathcal{F}_L(\mathcal{S}_i(t);\rho_i,m_i)\right], \quad \psi_i = \Psi_i(0).\label{eq:psi_i}
\end{equation}
This index $\Psi_i(t)$ represents the marginal long-term benefit of allocating an additional unit of traffic to creator $i$ at time $t$. The term $\mathcal{F}_L$ is the instantaneous follower growth rate, while $r/\delta$ converts the future revenue from these new followers into its present value. By \eqref{Eq:Bs_Evolution}, we note that $\Psi_i(t)$ decreases in $\mathcal{S}_i(t)$. The following proposition establishes the formal link between this economic concept and the policy structure in Theorem \ref{Thm:Opt_policy_2Crtr_Ln}.

\begin{proposition}\label{Prop:Opt_policy_2Crtr_Ln_Psi_Eq_gS1}
Under the linear follower growth model, for any time $t \geq 0$, the condition $\mathcal{S}_2(t) = g(\mathcal{S}_1(t))$ is equivalent to $\Psi_1(t) = \Psi_2(t)$, where $\Psi_i(t)$ is defined as in \eqref{eq:psi_i}. That is,
\begin{equation}
\mathcal{S}_2(t) = g(\mathcal{S}_1(t)) \quad \text{if and only if} \quad \Psi_1(t) = \Psi_2(t).\nonumber
\end{equation}
Moreover,
\begin{align}
    &\mathcal{S}_2(t) > g(\mathcal{S}_1(t)) \quad \text{if and only if} \quad \Psi_1(t) > \Psi_2(t);\nonumber\\
    &\mathcal{S}_2(t) < g(\mathcal{S}_1(t)) \quad \text{if and only if} \quad \Psi_1(t) < \Psi_2(t).\nonumber
\end{align}
\end{proposition}
Therefore, the optimal policy is a simple and intuitive rule: allocate all traffic to the creator with the higher marginal value $\Psi_i(t)$. If the values are equal, the platform allocates traffic to maintain this equality. This leads to a two-phase evolution of the optimal control trajectory:
\begin{enumerate}
    \item \textbf{(Prioritization Phase)} Initially, all traffic is allocated to the creator with the higher initial marginal value (say, creator $i_1$ with $\psi_{i_1}>\psi_{i_2}$) until their marginal value $\Psi_{i_1}(t)$ decreases to match the other's.

    \item \textbf{(Balanced Phase)} Once the marginal values are equalized, traffic is allocated to both creators in a way that maintains $\Psi_1(t)=\Psi_2(t)$ indefinitely. The system's state then evolves along the indifference line $\mathcal{S}_2=g(\mathcal{S}_1)$.
\end{enumerate}

\subsubsection{Illustration of the Optimal Policy.}\label{Subsubsec:Opt_policy_2Crtr_Ln_it}
In this section, we present an intuitive visualization of the optimal policy $\gamma^*$ and the resulting system trajectories. Figure~\ref{Fig:Opt_policy_2Crtr_Ln_illust} depicts the policy in the $(\mathcal{S}_1,\mathcal{S}_2)$ state space, which is partitioned by the indifference line $\mathcal{S}_2=g(\mathcal{S}_1)$. The system’s evolution is determined by its initial position relative to this line.

A trajectory starting at point $P_1$, above the line $\mathcal{S}_2=g(\mathcal{S}_1)$ corresponds to a state where creator 1 has the higher marginal value $(\psi_1>\psi_2)$. The platform allocates all traffic to creator 1, causing the state to move horizontally to the right until it reaches the indifference line.

Conversely, a trajectory starting at point $P_3$ below the line begins with $\psi_2>\psi_1$. All traffic is given to creator 2, causing the state to move vertically until it hits the line.

Once a trajectory reaches the indifference line (as from $P_1$ or $P_3$), or if it starts on the line (point $P_2$), the system enters the Balanced Phase. The platform then applies a continuous allocation that keeps the state on this line, with both follower bases growing simultaneously.

In summary, the indifference line $\mathcal{S}_2=g(\mathcal{S}_1)$ captures the essence of the optimal policy: it delineates the regions of exclusive traffic allocation and, once the marginal values $\Psi_1(t)$ and $\Psi_2(t)$ are equalized, it serves as the invariant path along which the system continues to evolve.

\begin{figure}[h]
\centering
\includegraphics[scale=0.5]{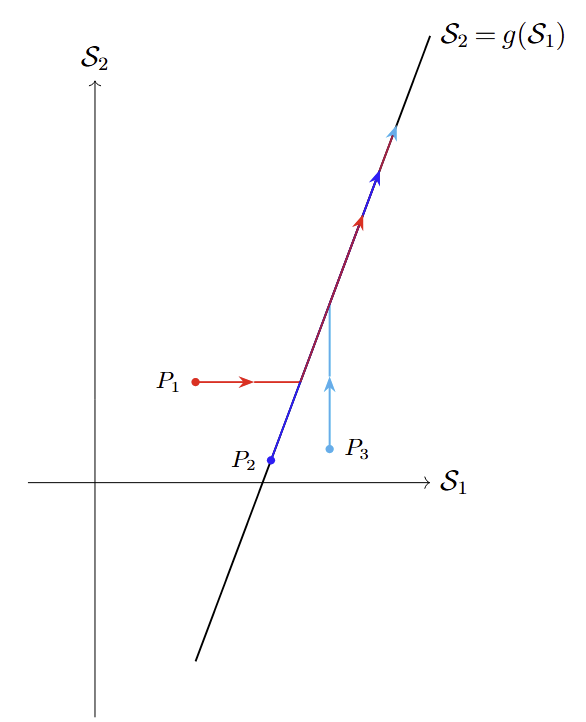}

\caption{The illustration of the optimal policy in $(\mathcal{S}_1,\mathcal{S}_2)$ state space with example trajectories shown, where $\brho_{1}=2, \brho_{2}=1.5, m_{1}=10, m_{2}=15, r=2, p=1, \delta=1$}
\label{Fig:Opt_policy_2Crtr_Ln_illust}
\end{figure}

\subsection{The Bass Model}\label{Subsec:Opt_policy_2Crtr_Bs}
We now extend the analysis to the Bass model $\mathcal{F}_B$, which incorporates the effect of word-of-mouth follower growth. The platform’s guiding principle of prioritizing the most valuable creator still applies, but the nonlinear dynamics give rise to a richer policy structure and deeper strategic insights. The optimal policy is derived in Section~\ref{Subsubsec:Opt_policy_2Crtr_Bs_op}, with the associated control evolution characterized, and the induced state trajectories illustrated in Section~\ref{Subsubsec:Opt_policy_2Crtr_Bs_it}.

\subsubsection{The Optimal Policy.}\label{Subsubsec:Opt_policy_2Crtr_Bs_op}
Under Bass growth, the optimal policy remains a bang-bang structure except on a boundary where interior sharing can be optimal. The policy is therefore described by state-dependent switching curves summarized in Theorem~\ref{Thm:Opt_policy_2Crtr_Bs} below.

\begin{theorem}\label{Thm:Opt_policy_2Crtr_Bs}
    Under the Bass follower growth with $2$ creators, let the initial number of followers for creator $i=1$, $2$ be $S_1$, $S_2$ respectively. Then, there exist functions $g_1, g_2, \bar{h}:\mathbb{R}\to\mathbb{R}$ such that the optimal traffic allocation policy $\gamma^*(\mathcal{S}_1,\mathcal{S}_2)=(\mathcal{A}_1^*,\mathcal{A}_2^*)$ takes the form
    \begin{align}
        \gamma^*=
        \begin{cases}
            (A,0)&\text{if }g_2(\mathcal{S}_2)\leq \mathcal{S}_1<g_1(\mathcal{S}_2),\\
            (0,A)&\text{if }\mathcal{S}_1<g_2(\mathcal{S}_2)\text{ or }\mathcal{S}_1>g_1(\mathcal{S}_2)\\
            &\text{or }\mathcal{S}_1=g_1(\mathcal{S}_2),\mathcal{S}_2\leq m_2(q-p\brho_2)/(2q),\\
            (\bar{h}(\mathcal{S}_2), A-\bar{h}(\mathcal{S}_2))&\text{if }\mathcal{S}_1=g_1(\mathcal{S}_2),
            \mathcal{S}_2> m_2(q-p\brho_2)/(2q).\nonumber
        \end{cases}
    \end{align}
    In particular, the function $\bar{h}$ admits the following closed-form expression, where $\mathcal{F}_{B,i}(\mathcal{S}_i):=\mathcal{F}_B(\mathcal{S}_i;\rho_i,m_i)$,
    \begin{align}
        \bar{h}(\mathcal{S}_2)&=A\cdot\Bigg[\frac{\frac{\rho_2}{m_2}\left(q-p\rho_2-\frac{2q}{m_2}\mathcal{S}_2\right)\mathcal{F}_{B,2}(\mathcal{S}_2)}{\frac{\rho_1}{m_1}\left(q-p\rho_1-\frac{2q}{m_1}g_1(\mathcal{S}_2)\right)\mathcal{F}_{B,1}(g_1(\mathcal{S}_2))+ \frac{\rho_2}{m_2}\left(q-p\rho_2-\frac{2q}{m_2}\mathcal{S}_2\right)\mathcal{F}_{B,2}(\mathcal{S}_2)}\Bigg].\nonumber
    \end{align}
\end{theorem}
Theorem~\ref{Thm:Opt_policy_2Crtr_Bs} partitions the state space using two (generally implicit) switching curves $\mathcal{S}_1=g_1(\mathcal{S}_2)$ and $\mathcal{S}_1=g_2(\mathcal{S}_2)$. Away from these curves, the optimal policy allocates all traffic to either creator 1 or creator 2. On the upper switching boundary $\mathcal{S}_1=g_1(\mathcal{S}_2)$, the platform follows a threshold rule in $\mathcal{S}_2$: it allocates all traffic to creator 2 when $\mathcal{S}_2\leq m_2(q-p\rho_2)/(2q)$, and otherwise splits traffic according to $\bar{h}(\mathcal{S}_2)$ once $\mathcal{S}_2>m_2(q-p\rho_2)/(2q)$.

To interpret these switching curves, it is useful to again introduce a marginal value index similar to linear growth case. Under Bass growth, the switching boundaries are no longer determined solely by a simple comparison of marginal values, as in the linear case. Nevertheless, the marginal value index remains a key component for understanding the structure of the optimal trajectory, and in particular the evolution within the balanced regime. We therefore define the marginal value index $\Phi_i(t)$, adapted to Bass growth dynamics, as follows:
\begin{equation}
\Phi_i(t) := \brho_i\left[1 + \frac{r}{\delta}\mathcal{F}_B(\mathcal{S}_i(t);\rho_i,m_i)\right], \quad \phi_i = \Phi_i(0).\label{eq:phi_i}
\end{equation}

However, unlike its linear counterpart, this marginal value is not always decreasing in the number of followers. The word-of-mouth effect introduces non-monotonicity.
\begin{lemma}\label{lem:2Crtr_bass_partial_Sit}
Under the Bass growth model, the partial derivative of the marginal value function $\Phi_i(t)$ defined in \eqref{eq:phi_i} with respect to $\mathcal{S}_i(t)$ satisfies
\begin{align}
    \frac{\partial \Phi_i(t)}{\partial \mathcal{S}_i(t)}=
    \begin{cases}
        \geq 0,&\text{if }\brho_i\leq\frac{q}{p}\text{ and }\mathcal{S}_i(t)\leq\frac{m_i(q-p\brho_i)}{2q},\\
        <0,&\text{otherwise.}
    \end{cases}\nonumber
\end{align}
Hence, $\Phi_i(t)$ increases only when both the creator capability and the follower base are at low levels. Once $\mathcal{S}_i(t)$ exceeds the threshold $m_i(q-p\brho_i)/(2q)$, $\Phi_i(t)$ will decrease with $\mathcal{S}_i(t)$.
\end{lemma}

The non-monotonicity in Lemma~\ref{lem:2Crtr_bass_partial_Sit} means that when a creator is still in the early acceleration stage (due to word-of-mouth effect), allocating additional traffic to that creator can increase its future marginal gains. Consequently, the optimality conditions extend beyond a static comparison of $\Phi_1$ and $\Phi_2$, giving rise to a conditional reversal phase in which all traffic is temporarily allocated to the creator with the currently lower marginal value in order to exploit the acceleration effect.

As a result, the first switching event need not occur at $\Phi_1=\Phi_2$. In particular, when creator 2 remains in the accelerating region, the platform may optimally switch from prioritizing creator 1 to creator 2 even when the gap $\Phi_1-\Phi_2$ is strictly positive. By contrast, once both creators exit the accelerating stage and the system enters the long-run sharing regime, the platform allocates traffic so as to maintain $\Phi_1=\Phi_2$. Thus, the condition $\Phi_1=\Phi_2$ characterizes the balanced regime in which neither creator is accelerating, but it does not fully describe the switching boundary, which needs to be described by the implicit curves $g_1$ and $g_2$.

Taking these dynamics into account, we characterize the optimal trajectory under Bass growth. We first introduce a key quantity governing switching behavior when a creator remains in the accelerating stage. In this region, the platform may optimally allocate traffic to a creator even when its marginal value is currently lower, in order to exploit the word-of-mouth acceleration effect. This behavior is characterized by an \textit{activation threshold}, denoted by $\eta_i\geq 0$, which represents the maximal marginal-value gap under which it becomes optimal to begin allocating traffic to creator $i$. Importantly, the activation threshold $\eta_i$ depends on the initial follower base $(S_1,S_2)$. For notational simplicity, we write $\eta_i$ as a shorthand, since the evolution of optimal trajectory is described for a fixed initial follower state.

We fix the initial marginal values $\Phi_i(0)=\phi_i$, and suppose $\phi_{i_1}>\phi_{i_2}$. Let $\tilde{S}_i=m_i(q-p\rho_i)/(2q)$, and denote $S_i=\mathcal{S}_i(0)$. The optimal trajectory then falls into the following cases.\newline

\noindent\textbf{Case 1. }When $S_{i_2}\geq\tilde{S}_{i_2}$, i.e., creator $i_2$ is already past the accelerating stage, the optimal trajectory follows a two-phase pattern:
\begin{itemize}
    \item[(a)] \textbf{(Prioritization Phase)} All traffic is allocated to creator $i_1$, until the marginal values equalize, i.e., $\Phi_{i_1}(t)=\Phi_{i_2}(t)$.
    \item[(b)] \textbf{(Balanced Phase)} Once this equality is reached, traffic is split to maintain $\Phi_{i_1}(t)=\Phi_{i_2}(t)$ thereafter.
\end{itemize}
\textbf{Case 2. }When $S_{i_2}<\tilde{S}_{i_2}$, this is the key region that we have insights unique to the Bass model, where creator $i_2$ remains in the accelerating region, the optimal trajectory may exhibit a conditional reversal phase. In this case, there exists a threshold $\eta_{i_2}\geq 0$ such that the evolution depends on the initial gap in marginal values.
\begin{itemize}
    \item[(2A)] If $\phi_{i_1}-\phi_{i_2}>\eta_{i_2}$, the optimal trajectory has three phases:
    \begin{itemize}
        \item[(a)] \textbf{(Prioritization Phase)} All traffic is allocated to creator $i_1$ until
        \begin{align*}
            \Phi_{i_1}(t)=\Phi_{i_2}(t)+\eta_{i_2}.
        \end{align*}
        \item[(b)] \textbf{(Conditional Reversal Phase)} Traffic is then temporarily entirely to creator $i_2$ until the marginal values equalize $\Phi_{i_1}(t)=\Phi_{i_2}(t)$.
        \item[(c)] \textbf{(Balanced Phase)} Then, the platform splits traffic to maintain $\Phi_{i_1}(t)=\Phi_{i_2}(t)$.
    \end{itemize}
    \item[(2B)] If $\phi_{i_1}-\phi_{i_2}\leq\eta_{i_2}$, the optimal trajectory begins directly with the conditional reversal:
    \begin{itemize}
        \item[(a)] \textbf{(Conditional Reversal Phase)} All traffic is allocated to creator $i_2$ until $\Phi_{i_1}(t)=\Phi_{i_2}(t)$.
        \item[(b)] \textbf{(Balanced Phase)} Then, the platform splits traffic to maintain $\Phi_{i_1}(t)=\Phi_{i_2}(t)$.
    \end{itemize}
\end{itemize}

\subsubsection{Illustration of the Optimal Policy.}\label{Subsubsec:Opt_policy_2Crtr_Bs_it}
In this section, we use Figure~\ref{Fig:Opt_policy_2Crtr_Bs_illust} to visualize the optimal policy under Bass dynamics as stated in Theorem ~\ref{Thm:Opt_policy_2Crtr_Bs}. The state space $(\mathcal{S}_1,\mathcal{S}_2)$ is partitioned by the two switching curves: the left curve $\mathcal{S}_1=g_2(\mathcal{S}_2)$ and the right curve $\mathcal{S}_1=g_1(\mathcal{S}_2)$. The horizontal dashed line $\mathcal{S}_2=m_2(q-p\rho_2)/(2q)$ marks the follower-base threshold that separates creator 2’s accelerating stage (below the line) from its decelerating stage (above the line), and therefore determines whether a conditional reversal can occur upon reaching the right switching boundary.

The system’s dynamics are richer than in the linear case. Consider a trajectory starting at $P_1$, which lies in a region where the policy allocates all traffic to creator 2. The state therefore moves vertically upward as $\mathcal{S}_2$ increases. When the trajectory reaches the left switching curve, the optimal action switches to allocating all traffic to creator 1, so the state moves horizontally to the right as $\mathcal{S}_1$ increases. Upon reaching the right switching curve, the trajectory enters the boundary regime, since $\mathcal{S}_2$ is above the dashed-line threshold, the platform splits traffic according to $\bar{h}(\mathcal{S}_2)$, and the state evolves upward while remaining aligned with the right switching curve.

Next, consider a trajectory starting at $P_3$, which lies between the two switching curves where the platform allocates all traffic to creator 1. The state initially moves horizontally to the right until it reaches the right switching curve. Since the intersection occurs below the dashed line, creator 2 is still in the accelerating stage, so the optimal policy induces a conditional reversal, temporarily allocating all traffic to creator 2. The state then moves vertically upward until it reaches the right switching boundary again at a higher $\mathcal{S}_2$, after which the system transitions into the balanced regime and traffic is split according to $\bar{h}(\mathcal{S}_2)$.

Finally, a trajectory starting at $P_2$ reaches the right switching curve above the dashed-line threshold. Since creator 2 is already past its accelerating stage, no conditional reversal is triggered, the platform moves directly into the balanced regime.

\begin{figure}[h]
\centering
\includegraphics[scale=0.5]{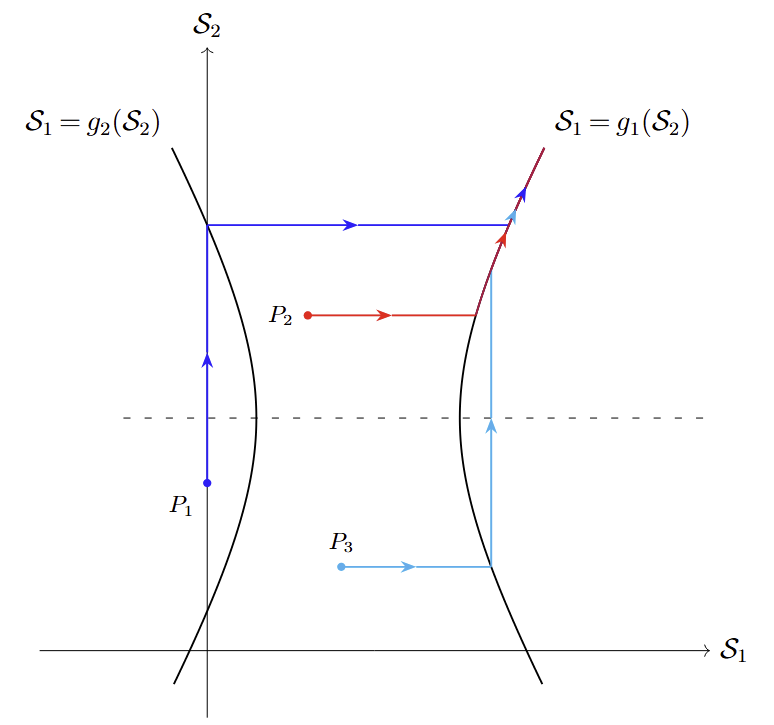}

\caption{The illustration of the optimal policy in $(\mathcal{S}_1,\mathcal{S}_2)$ with example trajectories shown, where $\brho_{1}=2.0, \brho_{2}=1.5, m_{1}=10.0, m_{2}=15.0, r=2, p=1, q=20, \delta=1$}
\label{Fig:Opt_policy_2Crtr_Bs_illust}
\end{figure}

In contrast to the linear model's simple path, the Bass model's non-monotonic marginal benefits create more complex trajectories with multiple switching phases, reflecting the platform's strategic response to word-of-mouth dynamics.

\subsection{Long-Term Follower Base}\label{Subsec:2Crtr_Long_Base}

Having characterized the optimal dynamic policies, we now address Research Question 2: how does the optimal policy shape the creator ecosystem in the long term? The ``most valuable creator first" rule not only governs momentary traffic allocation but also fundamentally determines the market structure. Specifically, it determines whether all creators will be ultimately nurtured by the platform (viability) and the effective limits imposed on their growth (disciplined growth). 

Our focus is on creator 2, since creator 1, being the more capable under the assumed ordering $\brho_1>\brho_2$, always receives sufficient traffic to eventually reach its full follower capacity $m_1$. 

As we will formalize in Proposition \ref{Prop:2Crtr_Long_base}, two key outcomes emerge. First, the platform performs an implicit capability screening: creator 2 receives positive traffic only if its capability $\brho_2$ exceeds a critical threshold $\brho_2^{\,\min}$, provided its initial follower base is sufficiently small. Second, even when this threshold is met, creator 2 grows only up to an effective long-term follower base $\bar{S}_2$, which is strictly below its natural capacity $m_2$. While $m_2$ is the theoretical maximum number of followers, $\bar{S}_2$ represents the actual asymptotic limit under the optimal policy and serves as the effective follower cap.


\begin{proposition}\label{Prop:2Crtr_Long_base}
    Assume $\rho_1> \rho_2$, with linear or Bass growth model and the optimal policy $\gamma^*$ stated in Theorem \ref{Thm:Opt_policy_2Crtr_Ln}, \ref{Thm:Opt_policy_2Crtr_Bs}, respectively,
    \begin{itemize}
        \item[(a)] Creator 1 achieves its full capacity asymptotically, $\lim_{t\to\infty}\mathcal{S}_1(t)=m_1$.
        \item[(b)] There exists threshold $\rho_{2}^{\min}$ such that creator 2 receives positive traffic if and only if $\rho_2>\rho_{2}^{\min}$. Explicitly, under linear growth we have
        \begin{align}
            \brho_2^{\min}=\brho_{2,linear}^{\,\min}&=
            \frac{-1+\sqrt{1+4\brho_1\cdot\frac{pr}{\delta}\left(1-\frac{S_2}{m_2}\right)}}{2\cdot\frac{pr}{\delta}\left(1-\frac{S_2}{m_2}\right)}.\nonumber
        \end{align}
        Under Bass growth we can bound the capability threshold from above
        \begin{align}
            \rho_2^{\,\min}\leq \brho_{2,bass}^{\,\min}&=
            \frac{-1+\sqrt{\left[1+\frac{rqS_2}{m_2\delta}\left(1-\frac{S_2}{m_2}\right)\right]^2+4\brho_1\cdot\frac{pr}{\delta}\left(1-\frac{S_2}{m_2}\right)}}{2\cdot\frac{pr}{\delta}\left(1-\frac{S_2}{m_2}\right)}-\frac{qS_2}{2pm_2}.\nonumber
        \end{align}
        Both $\brho_{2,linear}^{\,\min}$ and $\brho_{2,bass}^{\,\min}$ are decreasing in $p$ and $r$. Moreover, $\brho_{2,bass}^{\,\min}$ is decreasing in $q$. Hence, under a globally higher conversion rate or subscription fee, the admission requirement will be less strict.
        
        \item[(c)] Suppose that $\rho_2>\rho_2^{\min}$, the follower base of creator 2 converges to a finite upper bound $\bar{S}_2<m_2$. Explicitly, under linear growth
        \begin{align}
            \bar{S}_2=\bar{S}_{2,\text{linear}}=m_2\left(1-\frac{\brho_1-\brho_2}{\brho_2^2}\cdot\frac{\delta}{pr}\right).\nonumber
        \end{align}
        Under Bass growth we have
        \begin{align}
            \bar{S}_2=\bar{S}_{2,\text{bass}}=\frac{m_2}{2q}\left[-(p\brho_2-q)+\sqrt{(p\brho_2+q)^2-\frac{4q\delta}{r}\left(\frac{\brho_1}{\brho_2}-1\right)}\right].\nonumber
        \end{align}
        That is, $\mathcal{S}_2(t)\leq\bar{S}_2$ for all $t\geq 0$, and $\lim_{t\to\infty}\mathcal{S}_2(t)=\bar{S}_2=\sup_{t\geq 0}\mathcal{S}_2(t)$. 
        
        Both $\bar{S}_{2,\text{linear}}$ and $\bar{S}_{2,\text{bass}}$ are increasing in $\rho_2$, $p$, $r$. In addition, $\bar{S}_{2,\text{bass}}$ also increases in $q$. Thus, improving creator capability, conversion rate, subscription fee all raises creator 2’s long-term follower ceiling.
    \end{itemize}
\end{proposition}
This proposition reveals a deliberate and economically rational platform strategy. The platform does not passively allocate traffic, instead it actively screens creators, prioritizing only those with sufficient capability and growth potential. The threshold $\brho_2^{\,\min}$ acts as a barrier to entry, ensuring that traffic is not wasted on creators unlikely to generate a positive long-term return.


Furthermore, even for creators who pass this screen, the platform enforces disciplined growth. The effective cap $\bar{S}_2$ reflects the constant opportunity cost of diverting traffic from the more capable creator 1. This ensures that while promising creators are nurtured, the platform's resources remain optimally deployed to maximize total revenue.

\section{Optimal Policy: Multiple Creators}\label{Sec:Opt_policy_NCrtr}
In this section, we extend the framework from two to $N$ heterogeneous creators, focusing on the richer Bass dynamics because the linear model arises as a special case when $q=0$. Section~\ref{Subsec:Opt_policy_NCrtr_Bs} derives the optimal traffic allocation policy under $\mathcal{F}_B$, showing that the structural results from two-creator case follows. Importantly, the conditional reversal phenomenon observed in the two-creator Bass model also arises here, whenever a lagging creator remains in the early accelerating stage of their growth curve. Section~\ref{Subsec:Def_Crtr_type} introduces a classification of creator types, allowing the framework to extend further to multi-type environments. 


\subsection{The Bass Model}\label{Subsec:Opt_policy_NCrtr_Bs}
The optimal traffic allocation policy across all $N$ creators is characterized in Section~\ref{Subsubsec:Opt_policy_NCrtr_Bs_op}, together with a numerical example for $N=3$. In Section~\ref{Subsubsec:Opt_policy_NCrtr_Bs_ltb}, we further analyze the capability thresholds required for creators to receive positive traffic under the optimal policy and characterize their corresponding long-term follower bases.

    
\subsubsection{The Optimal Policy.}\label{Subsubsec:Opt_policy_NCrtr_Bs_op}
To characterize the optimal policy, we similarly use the marginal value functions $\Phi_i(t)$ and the activation thresholds $\eta_i$ defined as in Section~\ref{Subsubsec:Opt_policy_2Crtr_Bs_op}. For notational convenience, consider any set $M(t)\subseteq{1,\dots,N}$ such that $\Phi_i(t)=\Phi_{i'}(t)$ for all $i,i'\in M(t)$. We will write $\Phi_M(t):=\Phi_i(t)$ for any $i\in M(t)$. In addition, for any set $M(t)\subseteq{1,\dots,N}$ (with $\Phi_i(t)=\Phi_{i'}(t)$ for all $i,i'\in M(t)$) and any $j\notin M$, we will denote the gap $D_{M,j}(t)$ as
\begin{align*}
D_{M,j}(t):=\Phi_M(t)-\Phi_j(t).
\end{align*}
The following theorem presents the optimal traffic allocation policy.

\begin{theorem}\label{Thm:Opt_policy_NCrtr_Bs}
    Under the Bass follower growth with $N$ heterogeneous creators, let the initial number of followers for creator $i=1,\dots,N$ be $S_1,\dots,S_N$. Then, 
    \begin{itemize}
        \item i) There exist thresholds $\eta_1,\dots,\eta_N$ such that for each creator $j$: if $S_j<\tilde{S}_j$, then $\eta_j\geq 0$; if $S_j\geq\tilde{S}_j$, then $\eta_j=0$, where $\tilde{S}_j=m_j(q-p\rho_j)/(2q)$.
        \item ii) There exists a nonempty initial active set $M(0)\subseteq\{1,\dots,N\}$, where it can assumed without loss of generality that $M(0)=\{k\}$. 
        \item iii) $\Phi_k(0)-\Phi_j(0)=\phi_k-\phi_j>\eta_j$ for all $j\neq k$. 
    \end{itemize}
    The optimal policy is characterized by a sequence of switching times $\{t_{\ell}\}_{\ell\geq 0}$ and active sets $\{M_{\ell}\}_{\ell\geq 0}$ defined recursively. Set $t_0=0$, $M_0=M(0)$. Suppose at stage $\ell\geq 0$, the process begins at time $t_{\ell}$ with active set $M_{\ell}:=M(t_{\ell})$.
    \begin{itemize}
        \item[(1)] Fix $t_{\ell}$ and $M_{\ell}$ at the beginning of stage $\ell$, and let $\Phi_i(t)$ and $D_{M_{\ell},j}(t)$ be evaluated along the state trajectory specified by the allocation rule below:
        \begin{itemize}
            \item[(a)] If $M_{\ell}=\{k\}$, choose the allocation $\mathcal{A}_k(t)=A$, $\mathcal{A}_j(t)=0$ for all $j\neq k$, until the first time $$\tau_{\ell}= \inf\{t\geq t_{\ell}:\exists\,j\neq k\text{ such that }\Phi_k(t)-\Phi_j(t)\leq\eta_j\}.$$
            
            \item[(b)] If $|M_{\ell}|\geq 2$, allocate traffic on $M_{\ell}$ such that $\Phi_i(t)=\Phi_{i'}(t)$ for all $i,i'\in M_{\ell}$ until the first time $$\tau_{\ell}=\inf\{t\geq t_{\ell}: \exists\,j\notin M_{\ell}\text{ such that }D_{M_{\ell},j}(t)\leq\eta_j\}.$$
            In particular, if $\tau_{\ell}=+\infty$, then the allocation continues indefinitely and the process terminates.
        \end{itemize}
        \item[(2)] Let $j_{\ell}$ be any creator attaining the switching condition at time $\tau_{\ell}$.
        \begin{itemize}
            \item[(a)] If $\mathcal{S}_{j_{\ell}}(\tau_{\ell})\geq\tilde{S}_{j_{\ell}}$, set $t_{\ell+1}:=\tau_{\ell}$, proceed to (3).
            \item[(b)] If $\mathcal{S}_{j_{\ell}}(\tau_{\ell})<\tilde{S}_{j_{\ell}}$, trigger the conditional reversal phase, allocate all traffic to creator $j_{\ell}$ until the first equality time between $\Phi_{j_{\ell}}(t)$ and $\Phi_{M_{\ell}}(t)$ on the deceleration stage of creator $j_{\ell}$, $$\sigma_{\ell}=\inf\{t\geq \tau_{\ell}: \Phi_{j_{\ell}}(t)=\Phi_{M_{\ell}}(t)\},$$
            where $\Phi_i(t)$ for all $i$ evolve according to the allocation $\mathcal{A}_{j_{\ell}}(t)=A$.
            Define $t_{\ell+1}:=\sigma_{\ell}$, proceed to (3).
        \end{itemize}
        \item[(3)] Enlarge the active set $M_{\ell+1}\leftarrow M_{\ell}\cup\{j_{\ell}\}$, and at $t_{\ell+1}$, traffic is allocated on $M_{\ell+1}$ so that
            \begin{align*}
                \Phi_i(t_{\ell+1})=\Phi_{i'}(t_{\ell+1})\text{ for all }i,i'\in M_{\ell+1},
            \end{align*}
        \item[(4)] Then, we proceed to stage $\ell+1$ that begins at time $t_{\ell+1}$ with updated active set $M_{\ell+1}$ and repeat the above steps.
    \end{itemize}
\end{theorem}

To describe the optimal control evolution intuitively, the policy operates recursively. Starting from an initial active set, the platform maintains equality of marginal values among the active creators while continuously monitoring the inactive ones. Whenever the gap between the active set's marginal value and an inactive creator $j$'s marginal value shrinks to the activation threshold $\eta_j$, a transition occurs. The exact nature of this transition depends on the lagging creator's momentum: if creator $j$ has already passed their accelerating growth stage (i.e., $S_j(t) \ge \tilde{S}_j$), they are seamlessly added to the active set and sharing begins. However, if creator $j$ is still in their accelerating phase ($S_j(t) < \tilde{S}_j$), the platform triggers a temporary ``Conditional Reversal Phase,'' allocating all traffic exclusively to creator $j$ until their marginal value catches up to the active set. The process then continues recursively with the expanded active set. We present Example \ref{Ex:Opt_policy_NCrtr_Bs} to further illustrate this evolution for the $N=3$ case as below.

\begin{example}\label{Ex:Opt_policy_NCrtr_Bs}
We consider a numerical example with three heterogeneous creators. Let the parameters be given by $(\brho_1,\brho_2,\brho_3)=(3,2,1.5)$, $(m_1,m_2,m_3)=(2,4,2.5)$, $p,q,r,\delta = 1,5,1,1$, $(S_1,S_2,S_3)=(1.8,1.2,0.2)$. We choose these parameters so that all phases of the optimal policy arise.

We first compute the initial marginal values, $(\phi_1,\phi_2,\phi_3)=(5.25,7.80,7.482)$, which yields the ordering $\phi_2>\phi_3>\phi_1$, implying $(i_1,i_2,i_3)=(2,3,1)$. We also solve the switching thresholds $\eta_1,\eta_3$ who satisfy $\phi_2-\phi_1>\eta_1$, $\phi_2-\phi_3>\eta_3$, so the conditional reversal phases at not triggered for creator 1 or 3 at the beginning.

The optimal allocation policy evolves as follows:
\begin{enumerate}
    \item\textbf{(Initial Prioritization Phase)} All traffic is allocated to creator 2. During this phase, $\mathcal{S}_2(t)$ increases until $\Phi_2(t)$ decreases to $\Phi_3(t)+\eta_3=\phi_3+\eta_3$.

    \item\textbf{(Conditional Reversal Phase for Creator 3)} At the time when $\Phi_2(t)=\phi_3+\eta_3$, i.e., $\Phi_2(t)-\phi_3=\eta_3$, we evaluate whether creator 3 meets the reversal condition. Since $S_3=0.2<m_3(q-p\brho_3)/(2q)=0.875$, the platform temporarily reallocates all traffic to creator 3. During this period, $\Phi_3(t)$ will first increase then decrease back to $\phi_3$.

    \item\textbf{(Balancing Phase Between 2 and 3)} After $\Phi_3(t)$ declines back to $\phi_3$, the platform allocates traffic between creators 2 and 3 such that $\mathcal{A}_2^*(t)+\mathcal{A}_3^*(t)=A$, $\Phi_2(t)=\Phi_3(t)$, and these two marginal values co-evolve downward over time.

    \item\textbf{(Conditional Reversal Phase skipped for Creator 1)} When $\Phi_2(t)=\Phi_3(t)=\phi_1+\eta_1$, i.e., $\Phi_2(t)-\phi_1=\eta_1$, we evaluate whether creator 1 meets the reversal condition. Since $S_1=1.8>m_1(q-p\brho_1)/(2q)=0.4$, the condition is not met, the platform proceeds directly to the final allocation phase.

    \item\textbf{(Final Balanced Phase)} The platform allocates traffic among all three creators to maintain equality in marginal values, $\sum_{i=1}^3\mathcal{A}_i^*(t)=3$, $\Phi_1(t)=\Phi_2(t)=\Phi_3(t)$.
\end{enumerate}
\end{example}

We highlight three key features of the optimal allocation policy in Theorem~\ref{Thm:Opt_policy_NCrtr_Bs} and Example~\ref{Ex:Opt_policy_NCrtr_Bs}.  First, the policy can be characterized by the marginal value function $\Phi_i(t)$ and the thresholds $\eta_i$, which capture both the immediate payoff from an additional unit of traffic and its future impact through follower accumulation. This forward-looking metric contrasts with heuristic or myopic rules that focus only on short-term gains while ignoring discounted future revenue.  

Second, the policy alternates between two regimes: (i) exclusive allocation, where all traffic is directed to a single creator (or a set with equal marginal values), and (ii) balancing, where traffic is shared to maintain equality of $\Phi_i(t)$ within the active group. For instance, in Step~1 of Example~\ref{Ex:Opt_policy_NCrtr_Bs}, all traffic goes to creator~3 to exploit its high growth potential, while in Step~5 traffic is split among all three creators to satisfy $\Phi_1(t)=\Phi_2(t)=\Phi_3(t)$. Unlike static proportional rules, this structure coordinates short-term allocations with long-term platform value creation.

Third, under Bass dynamics the policy may temporarily reverse allocation toward a lagging creator if it remains in the accelerating stage of its growth curve. This conditional reversal highlights the strategic value of emphasizing creators who would otherwise be overlooked. We note that in the special case of multi-creator linear growth, conditional reversal is absent as $q=0$ results in $\tilde{S}_j = -\infty$ for all $j$ and the condition $S_j<\tilde{S}_j$ is never satisfied.

\subsubsection{Long-Term Follower Base.}\label{Subsubsec:Opt_policy_NCrtr_Bs_ltb}
Building on the two-creator analysis, we now extend our answer to Research Question 2 to the general setting of $N$ heterogeneous creators under Bass dynamics. We confirm that the optimal policy's role as a ``selective gatekeeper" and enforcer of ``disciplined growth" scales to the entire ecosystem. The platform does not simply promote all creators; instead, it enforces a strict hierarchy of viability and growth limits across creators.

Similarly as in the two-creator case, the most capable creator (indexed as 1) always receives sufficient traffic to reach its full capacity $m_1$. For all other creator $i$, there exists a threshold $\brho_i^{\min}$ such that creator $i$ receives positive traffic under the optimal policy if and only if $\brho_i>\brho_i^{\min}$. In that case, $i$ accumulates followers and converges to a long-term base $\bar{S}_i$. We summarize these results in the following proposition.

\begin{proposition}\label{Prop:NCrtr_Bs_Long_base}
    Assume $\rho_1> \rho_2>\dots> \rho_N$, with the Bass growth model, and the optimal policy $\gamma^*$ stated is Theorem \ref{Thm:Opt_policy_NCrtr_Bs},
    \begin{itemize}
        \item[(a)] Creator 1 achieves its full capacity asymptotically, $\lim_{t\to\infty}\mathcal{S}_1(t)=m_1$.
        \item[(b)] There exists a threshold $\brho_{i}^{\,\min}>0$ for each creator $i\in\mathcal{N}\setminus\{1\}$ such that creator $i$ receives positive traffic if and only if $\brho_i>\brho_{i}^{\,\min}$. In specific, we can bound $\brho_i^{\,\min}$ from above by
        \begin{align*}
            \brho_{i}^{\,\min}\leq \rho_{i,bass}^{\,\min}=
            \frac{-1+\sqrt{\left[1+\frac{rqS_i}{m_i\delta}\left(1-\frac{S_i}{m_i}\right)\right]^2+4\brho_1\cdot\frac{pr}{\delta}\left(1-\frac{S_i}{m_i}\right)}}{2\cdot\frac{pr}{\delta}\left(1-\frac{S_i}{m_i}\right)}-\frac{qS_i}{2pm_i}.\nonumber
        \end{align*}
        
        \item[(c)] For any creator $i\in\mathcal{N}\setminus\{1\}$ such that $\brho_i>\brho_i^{\,\min}$, the follower base of creator $i$ converges to a finite upper bound $\bar{S}_i<m_i$,
        \begin{align}
            \bar{S}_i=\frac{m_i}{2q}\left[-(p\brho_i-q)+\sqrt{(p\brho_i+q)^2-\frac{4q\delta}{r}\left(\frac{\brho_1}{\brho_i}-1\right)}\right].\nonumber
        \end{align}
        That is, $\mathcal{S}_i(t)\leq\bar{S}_i$ for all $t\geq 0$, and $\lim_{t\to\infty}\mathcal{S}_i(t)=\bar{S}_i=\sup_{t\geq 0}\mathcal{S}_i(t)$. 
        
        \item[(d)] Analogous to Proposition \ref{Prop:2Crtr_Long_base}, for each $i\neq 1$, the upper bound for the threshold, $\brho_{i,bass}^{\min}$ is decreasing in $p,r,q$. The long-term follower base $\bar{\mathcal{S}}_{i}$ is increasing in $\brho_i,p,r,q$. These monotonicity properties extend directly from the two-creator case with Bass growth.

    \end{itemize}
\end{proposition}


\subsection{Extension to Multiple Creator Types}\label{Subsec:Def_Crtr_type}
While our analysis so far has treated each creator as unique, real-world platforms often manage creators in cohorts or types that share similar characteristics (e.g., new gaming streamers or established cooking channels). To capture this, we extend the framework to a multi-type model. Each type is aggregated into a single representative creator, and the policies from Section~\ref{Subsec:Opt_policy_NCrtr_Bs} apply directly. Once a type’s total traffic share is determined, it is allocated uniformly across all creators of that type. We begin with the formal definition of a creator type.

\begin{definition}\label{Def:Crtr_type}
    We say that two creators $j_1,j_2$ belong to the same \emph{creator type} if they share identical model parameters and initial conditions. That is,
    \begin{align}
        \brho_{j_1}=\brho_{j_2},\quad m_{j_1}=m_{j_2},\quad \text{and }S_{j_1}=S_{j_2}.\nonumber
    \end{align}
    Creators that satisfy this condition are referred to as \emph{type-$j$} creators, and we denote their common parameters by $\brho_{(j)}$, $m_{(j)}$ and $S_{(j)}$, respectively. The number of creators belonging to type $j$ is denoted by $n_j$.
\end{definition}

If the platform hosts creators from $K$ distinct types, indexed by $j\in\mathcal{K}=\{1,\dots,K\}$, then the population is partitioned into $K$ disjoint groups with corresponding type-specific parameters $\{\brho_{(j)},m_{(j)},S_{(j)}\}_{j=1}^K$, and the total number of creators satisfies $\sum_{j=1}^Kn_j=N$.

The corresponding optimal policy follows a simple three-step procedure. First, the platform aggregates all $n_j$ creators of a given type $j$ into a single representative “meta creator.” Second, it applies the optimal $N$-creator policy from Theorem~\ref{Thm:Opt_policy_NCrtr_Bs} to these $K$ meta creators to determine the total traffic share $\mathcal{A}^*_{(j)}(t)$ for each type. Finally, the traffic allocated to type $j$ is distributed uniformly among the $n_j$ creators of that type, so that each individual creator receives $\mathcal{A}^*_{(j)}(t)/n_j$.

This approach allows the platform to simplify a high-dimensional problem of managing $N$ individuals into a more tractable problem of managing $K$ types, all while preserving the core logic of the optimal allocation strategy. 


\section{Numerical Results: The Heuristics}\label{Sec:Num_results}

To assess the benefits of the optimal allocation, we compare it with three heuristic benchmarks: the myopic policy, the follower-base adjusted policy, and the follower-growth adjusted policy. Numerical experiments are conducted on the two-creator Bass model, with parameters chosen so that both creators start with positive follower bases and sufficient growth potential. Our purpose is twofold: first, to identify settings in which the optimal policy yields substantial gains over heuristics, and to determine when simpler rules can approximate optimal performance closely enough for practical implementation (RQ3); second, to analyze how these policies differentially shape the long-term creator ecosystem (RQ2).

\subsection{The Heuristic Policies}
To evaluate the optimal allocation in practice, we benchmark it against three simple heuristics that capture common rules or tractable approximations to the optimal index. Each heuristic defines a simple ranking function and allocates all traffic to the highest-ranked creator (or divides uniformly if several tie).

\subsubsection*{Myopic Policy.}
The simplest rule is to ignore dynamics and allocate all resources to the creator with the highest capability $\brho_i$. This corresponds to ranking creators solely by their short-term revenue performance (per traffic unit) and directing all traffic to the top performer, i.e., creator $1$ (recall that $\rho_1>\rho_2$ by assumption).

\subsubsection*{Follower-Base Adjusted Policy (FBA).}
The second heuristic augments the capability measure by incorporating the creator’s current follower base, reflecting the idea that popularity today signals potential future value. We will rank the creators by a marginal value proxy
\begin{align*}
    \Phi_i^{(\text{FBA})}=\brho_i[1+\hat{r}\cdot \mathcal{S}_i(t)],
\end{align*}
where $\hat{r}$ is a weight parameter. Whenever a creator's score is higher, the platform allocates all traffic to that creator. This continues until the two scores become equal, after which traffic is allocated in proportions that maintain this equality thereafter.

The value of $\hat{r}$ is chosen from an admissible set $\mathcal{R}$ to maximize the NPV generated by this heuristic. Positive or negative values of $\hat{r}$ tilt the allocation toward creators with larger or smaller follower bases, respectively. In particular, when $\hat{r}\ge 0$, the FBA heuristic coincides with the myopic heuristic, whereas as $\hat{r}\to-\infty$ it degenerates into a rule that prioritizes the type with the smaller value of $\rho_i\mathcal{S}_i(t)$.

\subsubsection*{Follower-Growth Adjusted Policy (FGA).}
The third heuristic is motivated by the marginal value function $\Phi_i(t)$ defined in \eqref{eq:phi_i}. To preserve the forward-looking logic of $\Phi_i(t)$ while making the policy implementable in practice, we rank creators using the score
\[
\Phi_i^{(\text{FGA})}(t)=\rho_i\!\left[1+\kappa\cdot \mathcal{W}_i(t)\right],
\]
where (i) $\kappa$ is an approximation of $r/\delta$ and can be calibrated from historical subscription and discount rate data; and (ii) $\mathcal{W}_i(t)$ is an empirical measure of follower growth per unit of traffic, $(d\mathcal{S}_i/dt)/\mathcal{A}_i$.

For ease of implementation, we further simplify the policy by setting all activation thresholds to zero, i.e., $\eta_i=0$ for all $i$. Under this implementation, the platform simply allocates traffic and expand the activation set according to the ranking induced by $\Phi_i^{(\text{FGA})}$.

\subsection{Revenue Performance (RQ3)}
To evaluate the three heuristics, we measure their performance against the optimal policy using the NPV of total revenue, where we solve the optimal NPV by brute force motivated by the proven structural results.
\begin{definition}
    We define the \textit{gap} between the NPV generated by the optimal policy and some heuristic as $\text{NPV}_{\text{optimal}}-\text{NPV}_{\text{heuristic}}$, and the \textit{percentage gap} between the NPV generated by the optimal policy and some heuristic as $(\text{NPV}_{\text{optimal}}-\text{NPV}_{\text{heuristic}})/\text{NPV}_{\text{heuristic}}$.
\end{definition}

Our first experiments compare the myopic policy, the follower-base adjusted policy, and the optimal allocation. Figure \ref{Fig:heu_myopic_naive_pct} shows the percentage gap as the capability difference varies. The myopic rule can lose 10-25\% of NPV relative to optimal, especially when creators are similarly capable, as shown in Figure \ref{Fig:heu_myopic_naive_pct}(a). The follower-base adjusted rule can close this gap to 5-10\%  if the weight $\hat{r}$ is chosen correctly, as shown in Figure \ref{Fig:heu_myopic_naive_pct}(b).

\begin{figure}[h]
\centering
\begin{subfigure}[b]{0.49\textwidth}
    \includegraphics[scale=0.60]{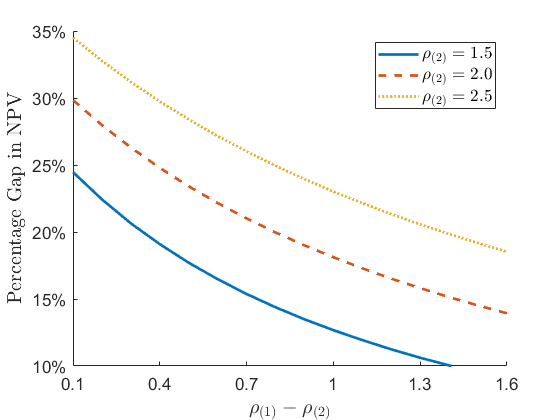}
\end{subfigure}
\begin{subfigure}[b]{0.49\textwidth}
    \includegraphics[scale=0.60]{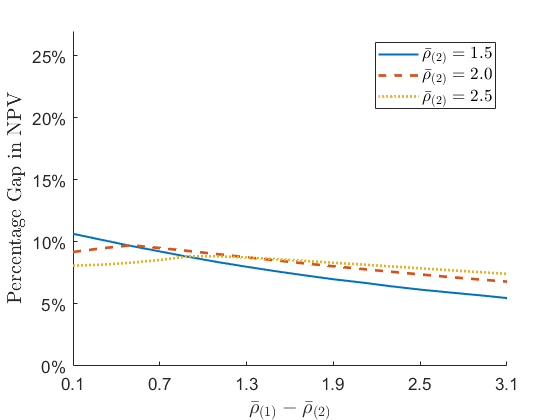}
\end{subfigure}
 \caption{Percentage Gap between the NPV generated by optimal and myopic policy (left) / follower-base adjusted policy (right), where $(A,S_{(1)},S_{(2)})=(10,1,0)$, $(n_1,n_2)=(5,5)$, $(m_{(1)},m_{(2)})=(10.0,15.0)$, $(\delta,q,p,r)=(1,20,1,2)$}
 \label{Fig:heu_myopic_naive_pct}
\end{figure}

The surprising and counter-intuitive result is that the optimal weight is always negative. While common intuition suggests giving more traffic to creators with more followers, the numerical results show the opposite: Performance improves when creators with smaller follower bases are ``promoted.'' The reason is that the marginal value of traffic decreases as a creator accumulates followers, so directing traffic toward smaller creators during their accelerating growth stage produces greater long-run returns. The negative weight in the heuristic effectively captures this dynamic, allowing it to mimic the optimal policy.

To dive in, we compare the performance gap between the myopic and follower-base adjusted policies, measured as $(\text{NPV}_{\text{follower base}}-\text{NPV}_{\text{myopic}})/\text{NPV}_{\text{myopic}}$, across different values of $\hat{r}$. This comparison further highlights the suboptimality that arises when $\hat{r}$ is chosen to be positive. As shown in Figure \ref{Fig:myopic_fba_gap}, the advantage of the follower-base adjusted policy over the myopic heuristic increases with $\hat{r}$ until reaching the optimal value. When $\hat{r} \ge 0$, however, the two heuristics coincide, yielding identical performance. On the other hand, setting $\hat{r}$ too low would also undermine the performance of the follower-based adjusted policy, as it would end up prioritizing the creator type with the smaller value of $\brho_i \mathcal{S}_i(t)$.

\begin{figure}[h]
    \centering
    \includegraphics[scale=0.6]{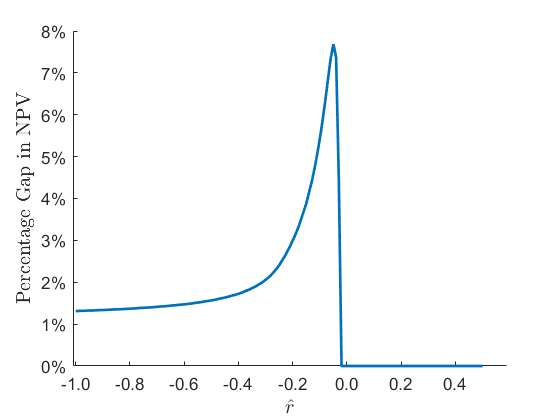}
    \caption{Percentage gap in NPV between myopic and follower-base adjusted policy, over different values of $\hat{r}$, where $(A,S_{(1)},S_{(2)})=(10,1,0)$, $(n_1,n_2)=(5,5)$, $(m_{(1)},m_{(2)})=(10.0,15.0)$, $(\delta,p,q,r)=(1,20,1,2)$}
    \label{Fig:myopic_fba_gap}
\end{figure}

This insight has practical value: platforms can estimate $\hat{r}$ from historical data and implement a simple follower-base adjusted rule to avoid being myopic. 

We next examine how the performance gaps depend on key economic parameters. Figures \ref{Fig:Myopic_delta_r_contour} and \ref{Fig:Myopic_p_q_contour} report contour plots of the NPV differences between the heuristics and the optimal policy over grids of $(\delta,r)$ and $(p,q)$ respectively.

Figure \ref{Fig:Myopic_delta_r_contour} shows that the gap is monotone in the ratio $\delta/r$. As the discount factor $\delta$ decreases or the subscription fee $r$ increases, the gap widens. The intuition is straightforward. Lower discounting increases the present value of future followers, while a higher $r$ raises the revenue from each follower. In both cases, long-run contributions become more valuable, so the forward-looking optimal policy has a larger advantage over myopic allocation. Although the pattern is similar for both heuristics, the follower-base adjusted policy yields smaller gaps than the purely myopic rule, reflecting its partial correction toward dynamic allocation.

Figure \ref{Fig:Myopic_p_q_contour} illustrates a similar monotonicity with respect to growth parameters. As either the direct conversion rate $p$ or the word-of-mouth parameter $q$ rises, the gap increases. Faster conversion means untapped populations are converted more quickly, and stronger word-of-mouth amplifies the value of early investments in follower accumulation. Both effects accentuate the forward-looking gains of the optimal policy relative to static heuristics.

\begin{figure}[h]
    \centering
    \begin{subfigure}[b]{0.49\textwidth}
        \includegraphics[scale=0.6]{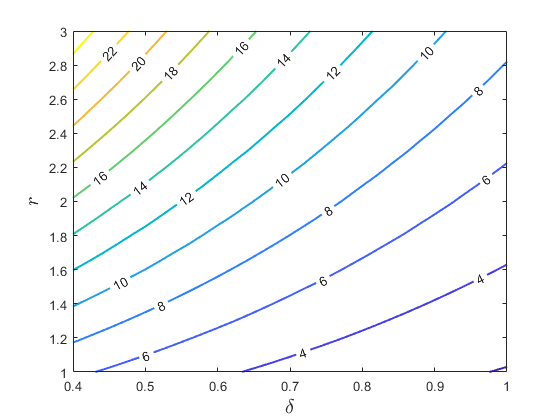}
    \end{subfigure}
    \begin{subfigure}[b]{0.49\textwidth}
        \includegraphics[scale=0.6]{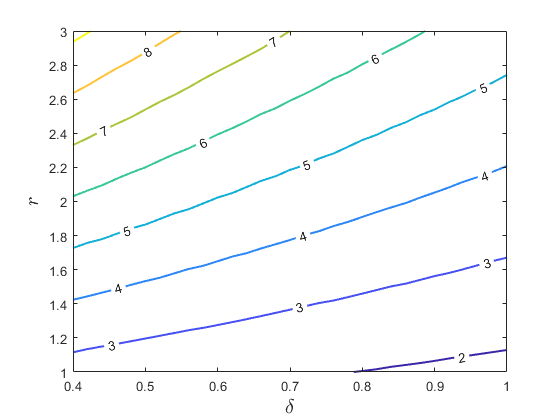}
    \end{subfigure}
 \caption{Gap between the NPV generated by optimal and myopic policy (left) / follower-base adjusted policy (right) over $\delta-r$ gird, where $(A,S_{(1)},S_{(2)})=(10,1,0)$, $(n_1,n_2)=(5,5)$, $(m_{(1)},m_{(2)})=(10.0,15.0)$, $(p,q)=(20,1)$}
 \label{Fig:Myopic_delta_r_contour}
\end{figure}

\begin{figure}[h]
    \centering
    \begin{subfigure}[b]{0.49\textwidth}
        \includegraphics[scale=0.6]{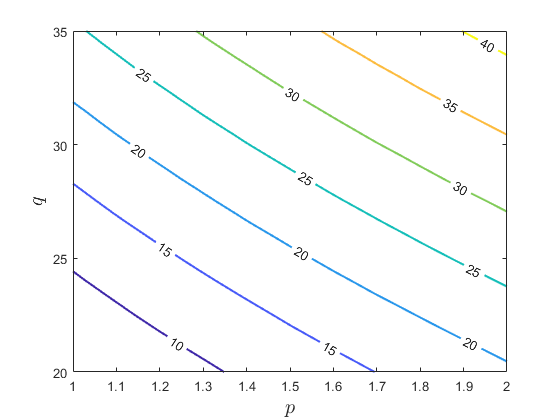}
    \end{subfigure}
    \begin{subfigure}[b]{0.49\textwidth}
        \includegraphics[scale=0.6]{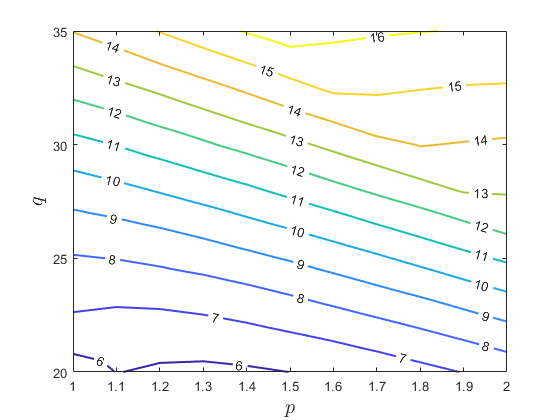}
    \end{subfigure}
 \caption{Gap between the NPV generated by optimal and myopic policy (left) / follower-base adjusted policy (right) over $p-q$ gird, where $(A,S_{(1)},S_{(2)})=(10,1,0)$, $(n_1,n_2)=(5,5)$, $(m_{(1)},m_{(2)})=(10.0,15.0)$, $(\delta,r)=(1,2)$}
 \label{Fig:Myopic_p_q_contour}
\end{figure}

Overall, these results indicate that the value of optimal control is greatest when the future is important (low $\delta$, high $r$) or when growth is rapid (high $p$, high $q$). In such regimes, forward-looking allocation delivers significantly higher revenue than heuristics, though the follower-base adjustment narrows the gap relative to a purely myopic rule.

Finally, we turn to the follower-growth adjusted policy, designed as a practical approximation to the optimal rule. This heuristic replaces the unknown ratio $r/\delta$ in the marginal value function with an empirical estimate $\kappa$. Figure \ref{Fig:robust} reports the percentage NPV gap between this heuristic and the optimal policy as $\kappa$ varies, where the red vertical dashed line marks the value $\kappa=r/\delta$.

Overall, the FGA policy performs close to the optimal benchmark across a wide range of $\kappa$, indicating strong robustness to misspecification. In the baseline case shown in Figure \ref{Fig:robust} (left), the performance gap remains negligible for most values of $\kappa$ and increases only when $\kappa$ is severely underestimated.

In contrast, Figure \ref{Fig:robust} (right) considers a more extreme environment in which the subscription fee $r$ is low and the capability gap between creators, $\rho_1-\rho_2$ is very large. In this case, the heuristic remains stable for moderate values of $\kappa$, but substantial overestimation can lead to a larger performance gap. Intuitively, when current monetization differences dominate and follower growth has relatively limited value, placing excessive weight on the growth component may distort the traffic allocation.

\begin{figure}[h]
    \centering
    \begin{subfigure}[b]{0.49\textwidth}
        \includegraphics[scale=0.32]{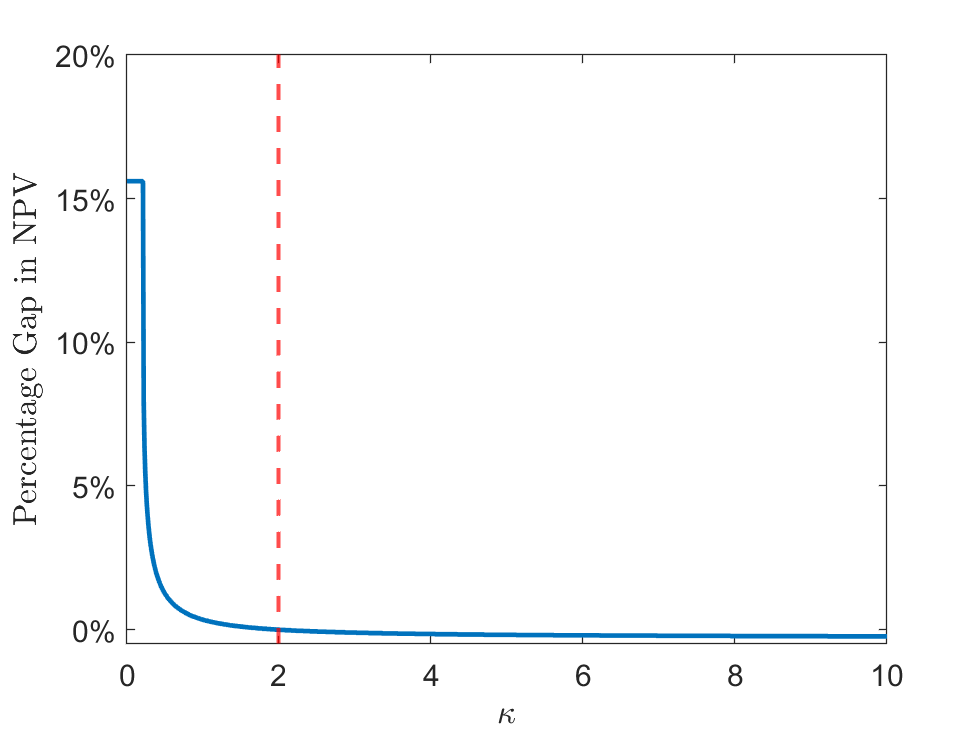}
    \end{subfigure}
    \begin{subfigure}[b]{0.49\textwidth}
        \includegraphics[scale=0.32]{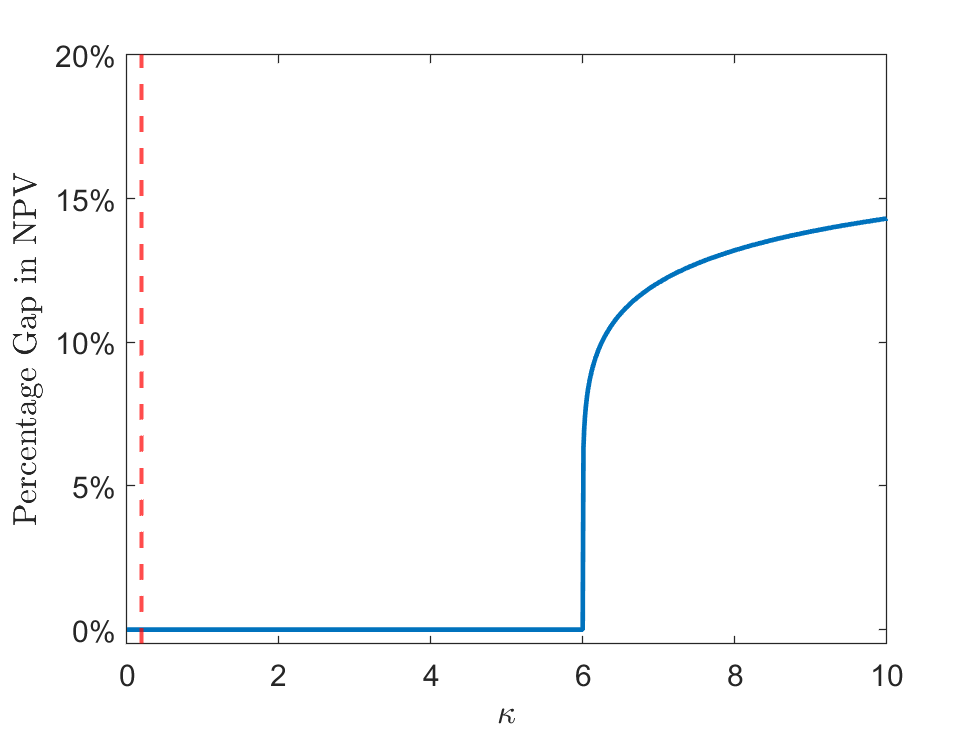}
    \end{subfigure}
 \caption{Percentage gap between optimal NPV and follower-growth adjusted policy against $\kappa$ when using parameters $(A,S_{(1)},S_{(2)})=(10,1,0)$, $(n_1,n_2)=(5,5)$, $(m_{(1)},m_{(2)})=(10.0, 15.0)$, $(\delta,q,p)=(1,20,1)$. In (left), $(\rho_1,\rho_2,r)=(2.0,1.5,2)$, in (right), $(\rho_1,\rho_2,r)=(15.0,1.5,0.2)$}
 \label{Fig:robust}
\end{figure}

Taken together, these results suggest that the FGA policy is broadly robust across different scenarios, with noticeable performance losses arising only under severe misspecification of $\kappa$.

\subsection{Ecosystem Implications: Selective vs. Indiscriminate Growth (RQ2)}

Beyond revenue, we now examine the long-term structural impact of these policies. Does the platform nurture a broad base of creators, or concentrate resources on a few stars? Figure \ref{Fig:fba_ltgr} illustrates the long-term follower bases of the two creators ($S_1$ and $S_2$) under the Follower-Base Adjusted (FBA) policy across different weight parameters $\hat{r}$, compared to the optimal policy benchmarks (dashed lines).

\begin{figure}[h]
    \centering
    \includegraphics[scale=0.75]{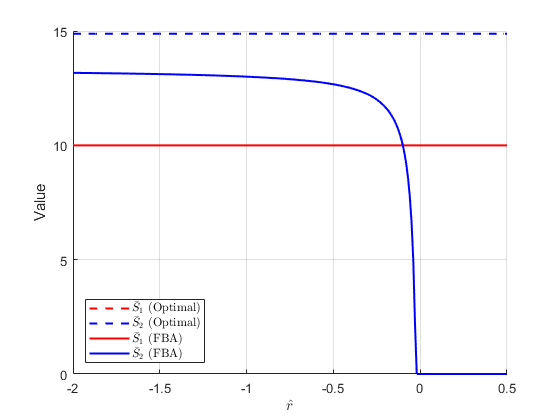}
 \caption{Parameters $(A,S_{(1)},S_{(2)})=(10,1,0)$, $(n_1,n_2)=(5,5)$, $(\brho_{(1)},\brho_{(2)},m_{(1)},m_{(2)})=(10.0, 15.0, 5.0,2.5)$, $(\delta,q,p,r)=(1,20,1,2)$.}
 \label{Fig:fba_ltgr}
\end{figure}

We observe three distinct regimes that highlight the ``selective" nature of the optimal policy:

1. Indiscriminate Growth (Very Negative $\hat{r}$): When the heuristic heavily penalizes large follower bases (left side of Figure 8), it effectively subsidizes the smaller Creator 2. Consequently, Creator 2 receives enough traffic to reach their full natural capacity ($m_2 = 15$). While this fosters a larger ``middle class," it is economically inefficient, as traffic is diverted from the high-capability star (Creator 1) well past the point of diminishing marginal returns.

2. Winner-Take-All (Positive $\hat{r}$): When $\hat{r} \geq 0$, the heuristic reinforces the advantage of the larger Creator 1. In this regime, Creator 2 is starved of traffic and remains at their initial low follower base ($S_2 \approx 0$). This mimics the outcome of the Myopic policy, failing to nurture potentially valuable emerging talent.

3. Selective Disciplined Growth (Optimal Policy): The optimal policy (dashed lines) strikes a precise middle ground. It supports Creator 2, but only up to an "effective ceiling" ($\bar{S}_2 \approx 10$) that is strictly lower than their natural capacity ($m_2 = 15$).

This result confirms that the optimal policy acts as a ``rational gatekeeper." Unlike the FBA heuristic, which tends to swing between the extremes of over-investment (indiscriminate support) and under-investment (winner-take-all), the optimal policy enforces disciplined growth. It ensures that emerging creators are nurtured only to the extent that their marginal long-term value exceeds that of the established star.

\section{Conclusion} 

This paper addresses the critical challenge platforms face in allocating viewer traffic to maximize revenue in a dual-revenue-stream creator economy. We develop a dynamic optimization model that captures the core trade-off between generating immediate advertising revenue and investing in long-term, follower-based monetization. Our analysis yields a characterization of the optimal policy, its long-term consequences for the creator ecosystem, and a practical, data-driven heuristic for implementation.

Our work provides several non-trivial and counter-intuitive insights. First, we find that the optimal policy follows a ``most-valuable-creator-first" rule driven by a forward-looking activation set. This policy is not myopic; under the influence of word-of-mouth effects, it can lead to a sophisticated ``conditional reversal" strategy, where it is optimal to temporarily re-invest all traffic in a lagging creator.

Second, regarding the ecosystem structure, the optimal policy acts as a selective gatekeeper, avoiding the extremes of myopic ``winner-take-all” markets or the inefficient “indiscriminate growth” produced by naive fairness heuristics. It establishes a strict capability threshold that serves as a barrier to entry. Furthermore, for those who pass this screen, it enforces disciplined growth by imposing an optimal follower ceiling. This ensures resources are not wasted on diminishing returns but are perpetually re-allocated to highest-potential assets.

Finally, our numerical analysis shows that while simple, myopic policies can lead to significant revenue loss, a practical, data-driven heuristic can achieve near-optimal results. We show that the common wisdom of ``betting on the winners" is flawed, and the best heuristics are those that run counter to this intuition. Specifically, a ``follower-growth adjusted" policy is remarkably robust to estimation errors and provides a clear path to superior performance.

Our findings offer several implications for both academic research and platform management. For theory, this paper contributes to the literature on platform management by providing a tractable dynamic model that bridges the gap between ad-based and subscription-based monetization systems. For managers, our work provides clear, actionable takeaways:

\begin{itemize}
\item
Shift from Static to Dynamic Metrics: Managers should move beyond simple, backward-looking metrics like raw follower counts and adopt forward-looking indices that capture growth potential.

\item
Beware of False Wisdom: The common strategy of simply rewarding past success can be highly costly. Our work shows that the best strategies are often counter-intuitive, designed to favor future growth potential over incumbency.

  \item The Reality of the Creator ``Middle Class'': While industry discourse hopes follower-based models will unlock a broad creator middle class, a revenue-maximizing platform acts as a strict gatekeeper. By enforcing a capability threshold, the platform curates a highly selective middle class rather than fostering indiscriminate growth.

\item
An Implementable Strategy: The follower-growth adjusted policy is our paper's core practical recommendation. It offers a transparent, robust, and data-driven rule that directly rewards a creator's empirically measured growth momentum, providing a near-optimal and readily implementable strategy.
\end{itemize}

This research can be extended in several directions. Our model makes several simplifying assumptions, each of which presents an opportunity for future inquiry. We assume that creator follower growth is independent; future work could model competitive or complementary dynamics between creators. We also assume that creator capability is fixed; a fruitful avenue would be to allow capability to evolve based on platform support or creator effort. Finally, our platform is a pure revenue-maximizer; future research could explore multi-objective platforms that also aim to optimize for goals like content diversity, novelty, or creator fairness.

%
%
%




\bibliographystyle{informs2014} 
\bibliography{ref_lit.bib} 

\ECSwitch


\ECHead{Proofs of Statements}


\section{Proof of Theorem \ref{Thm:Opt_policy_2Crtr_Ln}}
Under two creator case ($N=2$) and linear follower growth, we have
\begin{align*}
    \frac{d\mathcal{S}_1(t)}{dt}=\mathcal{A}_1(t)p\rho_1\left(1-\frac{S_1(t)}{m_1}\right),\;\frac{d\mathcal{S}_2(t)}{dt}=\mathcal{A}_2(t)p\rho_2\left(1-\frac{S_2(t)}{m_2}\right).
\end{align*}
The current-value Hamiltonian is
\begin{align}
    &\mathcal{H}(\mathcal{S}_1(t),\mathcal{S}_2(t),\mathcal{A}_1(t),\mathcal{A}_2(t),\lambda_1(t),\lambda_2(t),t)\nonumber\\
    &=\rho_1(\mathcal{A}_1(t)+r\mathcal{S}_1(t))+\rho_2(\mathcal{A}_2(t)+r\mathcal{S}_2(t))\nonumber\\
    &+\lambda_1(t)\mathcal{A}_1(t)p\rho_1\left(1-\frac{\mathcal{S}_1(t)}{m_1}\right)+\lambda_2(t)\mathcal{A}_2(t)p\rho_2\left(1-\frac{\mathcal{S}_2(t)}{m_2}\right).\label{eq:Ham_2Crtr_Linear}
\end{align}
Of which, $\lambda_1(t),\lambda_2(t)$ are the current-value costate variables that evolve according to
\begin{align*}
    \dot{\lambda}_i(t)=\delta \lambda_i(t)-\frac{\partial\mathcal{H}}{\partial\mathcal{S}_i}=\delta\lambda_i(t)-\rho_ir+\lambda_i(t)\mathcal{A}_i(t)\cdot\frac{p\rho_i}{m_i}\text{ for }i=1,2.
\end{align*}
Since $\mathcal{A}_1(t)+\mathcal{A}_2(t)=A$ for all $t$, we can substitute $\mathcal{A}_2(t)$ with $A-\mathcal{A}_1(t)$, then \eqref{eq:Ham_2Crtr_Linear} can be written as
\begin{align*}
    \left[\rho_1-\rho_2+\lambda_1(t)p\rho_1\left(1-\frac{\mathcal{S}_1(t)}{m_1}\right)-\lambda_2(t)p\rho_2\left(1-\frac{\mathcal{S}_2(t)}{m_2}\right)\right]\mathcal{A}_1(t)+C:=E(t)
    \mathcal{A}_1(t)+C.
\end{align*}
where $C$ is some term independent of $\mathcal{A}_1(t)$. Therefore, a necessary condition for $\mathcal{A}_1^*(t)$ with the corresponding state trajectory $\mathcal{S}_i^*(t)$ to be optimal is that there exist $\lambda_i(t)$ (for $i=1,2$) such that
\begin{align}
    \mathcal{A}_1^*(t)=\begin{cases}
        A&\text{if }E(t)>0,\\
        0&\text{if }E(t)<0,\\
        \text{any value in }[0,A]&\text{if }E(t)=0.
    \end{cases}\label{eq:Ln_PMP}
\end{align}
Since this is an infinite horizon free-end-point problem, we can impose the transversality condition
\begin{align*}
    \lim_{t\to\infty}e^{-\delta t}\lambda_i(t)=0,    
\end{align*}
combining this with the fact that the Hamiltonian $\mathcal{H}$ is linear in $\mathcal{S}$ at each $t$, and there is no terminal cost, the necessary condition \eqref{eq:Ln_PMP} becomes sufficient.

Differentiating $E(t)$ yields
\begin{align*}
    \dot{E}(t)=\sum_{i=1}^2 (-1)^{i-1}\left[p\rho_i\dot{\lambda}_i(t)-\frac{p\rho_i}{m_i}\left( \dot{\lambda}_i(t)\mathcal{S}_i(t)+\lambda_i(t)\dot{\mathcal{S}}_i(t)\right)\right].
\end{align*}
Of which,
\begin{align*}
    p\rho_i\dot{\lambda}_i(t)-\frac{p\rho_i}{m_i}\left[\dot{\lambda}_i(t)\mathcal{S}_i(t)+\lambda_i(t)\dot{\mathcal{S}}_i(t)\right]
    &=p\rho_i\left[\delta\lambda_i(t)-\rho_i r+\lambda_i(t)\mathcal{A}_i(t)\cdot\frac{p\rho_i}{m_i}\right]\\
    &-\frac{p\rho_i}{m_i}\mathcal{S}_i(t)\left[\delta\lambda_i(t)-\rho_i r+\lambda_i(t)\mathcal{A}_i(t)\cdot\frac{p\rho_i}{m_i}\right]\\
    &-\frac{p\rho_i}{m_i}\lambda_i(t)\mathcal{A}_i(t)\cdot p\rho_i\left(1-\frac{\mathcal{S}_i(t)}{m_i}\right)\\
    &=p\rho_i\left(1-\frac{\mathcal{S}_i(t)}{m_i}\right)(\delta\lambda_i(t)-r\rho_i).
\end{align*}
So we have
\begin{align*}
    \dot{E}(t)=\sum_{i=1}^2(-1)^{i-1}p\rho_i\left(1-\frac{\mathcal{S}_i(t)}{m_i}\right)(\delta\lambda_i(t)-r\rho_i).
\end{align*}
In particular, we notice that
\begin{align*}
    \dot{E}(t)-\delta E(t)=-pr\left[\rho_1^2\left(1-\frac{\mathcal{S}_1(t)}{m_1}\right)-\rho_2^2\left(1-\frac{\mathcal{S}_2(t)}{m_2}\right)+\frac{\delta}{pr}(\rho_1-\rho_2)\right]:=-prD(t).
\end{align*}
We now proceed to discuss different cases of initial state.

\subsection{Case 1: $S_2=g(S_1)$}\label{Subsec:Pf_Opt_policy_2Crtr_Ln_Case1}

If $S_2=g(S_1)$, we show that $\mathcal{A}_1(t)=h(\mathcal{S}_1(t))$ for all $t$ is the optimal policy. First, we plug $\mathcal{S}_2(t)=g(\mathcal{S}_1(t))$ to $D(t)$, which yields
\begin{align*}
    D(t)&=\rho_1^2\left(1-\frac{\mathcal{S}_1(t)}{m_1}\right)-\rho_2^2\left(1-\frac{g(\mathcal{S}_1(t))}{m_2}\right)+\frac{\delta}{pr}(\rho_1-\rho_2)\\
    &=\rho_1^2-\rho_2^2-\frac{\rho_1^2}{m_1}\mathcal{S}_1(t)+\frac{\rho_2^2}{m_2}\left[\frac{m_2\rho_1^2}{m_1\rho_2^2}\mathcal{S}_1(t)-\frac{m_2}{\rho_2^2}\left(\frac{\delta}{pr}(\rho_1-\rho_2)+\rho_1^2-\rho_2^2\right)\right]+\frac{\delta}{pr}(\rho_1-\rho_2)\\
    &=\rho_1^2-\rho_2^2-\frac{\rho_1^2}{m_1}\mathcal{S}_1(t)+\frac{\rho_1^2}{m_1}\mathcal{S}_1(t)-\left[\frac{\delta}{pr}(\rho_1-\rho_2)+(\rho_1^2-\rho_2^2)\right]+\frac{\delta}{pr}(\rho_1-\rho_2)\\
    &=0.
\end{align*}
Then, we show that enforcing $\mathcal{A}_1(t)=h(\mathcal{S}_1(t))$ is the necessary condition to maintain $\mathcal{S}_2(t)=g(\mathcal{S}_1(t))$ for all $t$. If $\mathcal{S}_2(t)=g(\mathcal{S}_1(t))$, then we have
\begin{align*}
    \dot{\mathcal{S}}_2(t)&=\frac{m_2\rho_1^2}{m_1\rho_2^2}\dot{\mathcal{S}}_1(t),\\
    \mathcal{A}_2(t)p\rho_2\left(1-\frac{\mathcal{S}_2(t)}{m_2}\right)
    &=\frac{m_2\rho_1^2}{m_1\rho_2^2}\mathcal{A}_1(t)p\rho_1\left(1-\frac{\mathcal{S}_1(t)}{m_1}\right)\\
    (A-\mathcal{A}_1(t))p\rho_2\left(1-\frac{\mathcal{S}_2(t)}{m_2}\right)&=\frac{m_2\rho_1^2}{m_1\rho_2^2}\mathcal{A}_1(t)p\rho_1\left(1-\frac{\mathcal{S}_1(t)}{m_1}\right)\\
    \mathcal{A}_1(t)&=A\cdot\left[\frac{\rho_2\left(1-\frac{\mathcal{S}_2(t)}{m_2}\right)}{\rho_2\left(1-\frac{\mathcal{S}_2(t)}{m_2}\right)+\frac{m_2\rho_1^3}{m_1\rho_2^2}\left(1-\frac{\mathcal{S}_1(t)}{m_1}\right)}\right]\\
    \mathcal{A}_1(t)&=A\cdot\frac{m_1\rho_2^3(m_2-\mathcal{S}_2(t))}{m_1\rho_2^3(m_2-\mathcal{S}_2(t))+m_2\rho_1^3(m_1-\mathcal{S}_1(t))}.
\end{align*}
By substituting $\mathcal{S}_2(t)$ with $g(\mathcal{S}_1(t))$, we recover $h(\mathcal{S}_1(t))$ in Theorem \ref{Thm:Opt_policy_2Crtr_Ln}. Therefore, given that $\mathcal{S}_2(0)=g(\mathcal{S}_1(0))$, by imposing $\mathcal{A}_1(t)=h(\mathcal{S}_1(t))$, $\mathcal{S}_2(t)=g(\mathcal{S}_1(t))$ for all $t$, which implies $D(t)\equiv 0$, $\dot{E}(t)-\delta E(t)=0$ for all $t$,
\begin{align*}
    e^{-\delta t}\dot{E}(t)&=\delta e^{-\delta t}E(t)\\
    \frac{d}{dt}(e^{-\delta t}E(t))&=0,
\end{align*}
suggesting that $e^{-\delta t}E(t)$ is a constant. We also note that
\begin{align*}
    \lim_{t\to\infty}e^{-\delta t}E(t)&=\lim_{t\to\infty}e^{-\delta t}(\rho_1-\rho_2)+e^{-\delta t}\left[\lambda_1(t)p\rho_1\left(1-\frac{\mathcal{S}_1(t)}{m_1}\right)-\lambda_2(t)p\rho_2\left(1-\frac{\mathcal{S}_2(t)}{m_2}\right)\right]\\
    &=0+p\rho_1\lim_{t\to\infty}e^{-\delta t}\lambda_1(t)\left(1-\frac{\mathcal{S}_1(t)}{m_1}\right)-p\rho_2\lim_{t\to\infty}e^{-\delta t}\lambda_2(t)\left(1-\frac{\mathcal{S}_2(t)}{m_2}\right)\\
    &=0
\end{align*}
as $\mathcal{S}_1(t),\mathcal{S}_2(t)$ are bounded by $m_1,m_2>0$, and $\lim_{t\to\infty}e^{-\delta t}\lambda_i(t)=0$ for $i=1,2$. Therefore, $e^{-\delta t}E(t)\equiv 0$, $E(t)\equiv 0$. Therefore, the allocation $\mathcal{A}_1(t)=h(\mathcal{S}_1(t))$ is optimal when we start on $\mathcal{S}_2=g(\mathcal{S}_1)$.

\subsection{Case 1: $S_2\neq g(S_1)$}

If the initial state $(S_1,S_2)$ is off the line $\mathcal{S}_2=g(\mathcal{S}_1)$, denote the deviation $\Delta(t):=\mathcal{S}_2(t)-g(\mathcal{S}_1(t))$. Note that
\begin{align*}
    D(t)&=\rho_1^2\left(1-\frac{\mathcal{S}_1(t)}{m_1}\right)-\rho_2^2\left(1-\frac{\mathcal{S}_2(t)}{m_2}\right)+\frac{\delta}{pr}(\rho_1-\rho_2)\\
    &=\frac{\rho_2^2}{m_2}(\mathcal{S}_2(t)-g(\mathcal{S}_1(t)),
\end{align*}
so the sign of $D(t)$ is equal to the sign of $\mathcal{S}_2(t)-g(\mathcal{S}_1(t))$, i.e., $\text{sign}(D(t))=\text{sign}(\Delta(t))$.

Assume $\Delta(0)>0$, consider the bang-bang control $\mathcal{A}_1(t)=A$, then $\mathcal{S}_2(t)$ will stay at $S_2$, $\mathcal{S}_1(t)$ will increase monotonically, since $g(\cdot)$ is linear with positive slope, $g(\mathcal{S}_1(t))$ strictly increases in $t$, therefore $\Delta(t)=S_2-g(\mathcal{S}_1(t))$ is strictly decreasing in $t$. Hence, there is at most one hitting time $\tau>0$ such that $\Delta(\tau)=0$.

Suppose $\tau>0$ is finite such that $\Delta(t)>0$ on $(0,\tau)$, $\Delta(\tau)=0$, we then have
\begin{align*}
    \dot{E}(t)-\delta E(t)&=-prD(t)\\
    &=-pr\cdot \frac{\rho_2^2}{m_2}\Delta (t).
\end{align*}
Multiply by $e^{-\delta t}$ and integrate from $t$ to $+\infty$, we obtain
\begin{align*}
    \frac{d}{ds}(e^{-\delta s}E(s))&=-pr\cdot\frac{\rho_2^2}{m_2}e^{-\delta s}\Delta (s)\\
    e^{-\delta t}E(t)&=pr\cdot \frac{\rho_2^2}{m_2}\int_t^{\infty}e^{-\delta s}\Delta (s)\,ds.
\end{align*}
By Theorem \ref{Thm:Opt_policy_2Crtr_Ln}, our proposed policy enforces $\Delta(s)>0$ on $s\in [t,\tau)$, and for $s\geq \tau$, $\Delta(s)\equiv 0$.

Hence, for any $t<\tau$, the integral $\int_t^{\infty}e^{-\delta s}\Delta(s)\,ds$ is strictly positive, suggesting that $E(t)>0$, so the bang-bang control $\mathcal{A}_1(t)=A$ on $(0,\tau)$ is optimal by \eqref{eq:Ln_PMP}. For time $s\geq\tau$, since $\Delta(s)\equiv 0$, $E(t)\equiv 0$ for $t\geq\tau$, by Section \ref{Subsec:Pf_Opt_policy_2Crtr_Ln_Case1} we know the control $\mathcal{A}_1(t)=h(\mathcal{S}_1(t))$ is optimal.

If there does not exist finite $\tau>0$ such that $\Delta(\tau)=0$, then $\mathcal{S}_2(t)>g(\mathcal{S}_1(t))$ for all $t$, same analysis applies where $\Delta(s)>0$ for any $s$, $E(t)>0$ for all $t$, and the optimal policy is $\mathcal{A}_1(t)\equiv A$.

Similarly, when $\Delta(0)<0$, the optimal policy is $\mathcal{A}_1(t)=0$ until $\mathcal{S}_2(t)=g(\mathcal{S}_1(t))$ is reached, from which the optimal policy switches to $\mathcal{A}_1(t)=h(\mathcal{S}_1(t))$.

\section{Proof of Proposition \ref{Prop:Opt_policy_2Crtr_Ln_Psi_Eq_gS1}}
Under the linear growth model, by \ref{eq:psi_i},
\begin{align*}
    \Psi_1(t)-\Psi_2(t)&=\rho_1-\rho_2+\frac{pr}{\delta}\left[\rho_1^2\left(1-\frac{\mathcal{S}_1(t)}{m_1}\right)-\rho_2^2\left(1-\frac{\mathcal{S}_2(t)}{m_2}\right)\right]\\
    &=\frac{pr}{\delta}\left[\frac{\delta}{pr}(\rho_1-\rho_2)+\rho_1^2-\rho_2^2-\frac{\rho_1^2}{m_1}\mathcal{S}_1(t)+\frac{\rho_2^2}{m_2}\mathcal{S}_2(t)\right].
\end{align*}
By the definition of $g(\cdot)$ in Theorem \ref{Thm:Opt_policy_2Crtr_Ln},
\begin{align*}
    \frac{\rho_2^2}{m_2}\left(\mathcal{S}_2(t)-g(\mathcal{S}_1(t))\right)
    =\frac{\delta}{pr}(\rho_1-\rho_2)+\rho_1^2-\rho_2^2-\frac{\rho_1^2}{m_1}\mathcal{S}_1(t)+\frac{\rho_2^2}{m_2}\mathcal{S}_2(t).
\end{align*}
Therefore,
\begin{align*}
    \Psi_1(t)-\Psi_2(t)=\frac{pr}{\delta}\cdot\frac{\rho_2^2}{m_2}\left(\mathcal{S}_2(t)-g(\mathcal{S}_1(t))\right).
\end{align*}
Since
\begin{align*}
    \frac{pr}{\delta}\cdot\frac{\rho_2^2}{m_2}>0,
\end{align*}
the sign of $\Psi_1(t)-\Psi_2(t)$ is exactly the sign of $\mathcal{S}_2(t)-g(\mathcal{S}_1(t))$, this completes the proof.\qed

\section{Proof of Lemma \ref{lem:2Crtr_bass_partial_Sit}}
Recall that
\begin{align*}
    \Phi_i(t)=\rho_i\left[1+\frac{r}{\delta}\left(p\rho_i+q\frac{\mathcal{S}_i(t)}{m_i}\right)\left(1-\frac{\mathcal{S}_i(t)}{m_i}\right)\right].
\end{align*}
Differentiating with respect to $\mathcal{S}_i(t)$ gives
\begin{align*}
    \frac{\partial\Phi_i(t)}{\partial\mathcal{S}_i(t)}&=\rho_i\cdot\frac{r}{\delta}\cdot\frac{\partial}{\partial \mathcal{S}_i(t)}\left[\left(p\rho_i+q\frac{\mathcal{S}_i(t)}{m_i}\right)\left(1-\frac{\mathcal{S}_i(t)}{m_i}\right)\right]\\
    &=\rho_i\cdot\frac{r}{\delta}\left[\frac{q-p\rho_i}{m_i}-\frac{2q}{m_i^2}\mathcal{S}_i(t)\right]\\
    &=\frac{\rho_i r}{\delta m_i}\left(q-p\rho_i-\frac{2q}{m_i}\mathcal{S}_i(t)\right).
\end{align*}
Since $\rho_i r/(\delta m_i)>0$, the sign of $\partial\Phi_i(t)/\partial \mathcal{S}_i(t)$ is the sign of
\begin{align*}
    q-p\rho_i-\frac{2q}{m_i}\mathcal{S}_i(t).
\end{align*}
Thus,
\begin{align*}
    \frac{\partial\Phi_i(t)}{\partial \mathcal{S}_i(t)}\geq 0\Leftrightarrow \mathcal{S}_i(t)\leq\frac{m_i(q-p\rho_i)}{2q}\text{ and }q-p\rho_i\geq 0
\end{align*}
which is equivalent to
\begin{align*}
    \rho_i\leq\frac{q}{p}\text{ and }\mathcal{S}_i(t)\leq\frac{m_i(q-p\rho_i)}{2q}.
\end{align*}
In all other cases, the derivative is strictly negative. Hence, $\Phi_i(t)$ can increase only when the creator is sufficient week $(\rho_i\leq q/p)$ or still has a sufficiently small follower base; once $\mathcal{S}_i(t)>m_i(q-p\rho_i)/(2q)$, $\Phi_i(t)$ is decreasing in $\mathcal{S}_i(t)$.\qed

\newpage
\section{Proof of Theorem \ref{Thm:Opt_policy_2Crtr_Bs}}
We proceed in two steps. First, we construct the two switching curves $g_1$ and $g_2$ and the boundary allocation $\bar{h}$ that appear in Theorem \ref{Thm:Opt_policy_2Crtr_Bs}. The construction uses the switching function and its integral representation to identify the switching boundaries, and then defines $\bar h$ on the upper switching boundary so that the balanced regime preserves $\Phi_1=\Phi_2$. Second, taking this candidate policy as given, we verify the sufficient conditions of the Pontryagin maximum principle: admissibility, concavity of the Hamiltonian in the state variables, and the maximization condition determined by the sign of the switching function. This establishes the policy stated in Theorem \ref{Thm:Opt_policy_2Crtr_Bs}.

\subsection{Construction of the Proposed Policy}
In this section we construct the proposed policy with definitions of $g_1$, $g_2$, $\bar{h}$ that we later prove its optimality. First, we define the following function
\begin{align}
    I(t)=\delta\int_t^{\infty}e^{\delta(t-s)}\left[\Phi_1(\mathcal{S}_1(s))-\Phi_2(\mathcal{S}_2(s))\right]\,ds.\label{eq:E_integral_representation}
\end{align}
We will use \eqref{eq:E_integral_representation} to define the switching curves that appear in the proposed policy.

When $\mathcal{A}_i(t)\equiv A$, (i.e., creator $i$ receives full traffic) with initial condition $x\in[0,m_i]$, we define $\Xi_i(t;x)$ as the unique solution to 
\begin{align}
    \frac{d}{dt}\Xi_i(t;x):=A\cdot\mathcal{F}_{B,i}(\Xi_i(t;x)),\quad \Xi_i(0;x)=x\in[0,m_i],\label{eq:full_alloc_flow}
\end{align}
that is, the state evolution of $\mathcal{S}_i(t)$ under full traffic allocation. Given a target level $y\in[0,m_i]$, we define the hitting time
\begin{align*}
    T_i(x,y):=\inf\{t\geq 0: \Xi_i(t;x)=y\}.
\end{align*}
Note that $\Phi_i(\mathcal{S}_i(t))$ is strictly increasing in $[0,\tilde{S}_i]$, and strictly decreasing on $[\tilde{S}_i,m_i]$, we can define the inverse branches of $\Phi_1$ by
\begin{align*}
    \zeta_1^L&:\Phi_1([0,\tilde{S}_1])\to [0,\tilde{S}_1],\\
    \zeta_1^H&:\Phi_1([\tilde{S}_1,m_1])\to [\tilde{S}_1,m_1],
\end{align*}
where $\zeta_1^{L}(u)$ and $\zeta_1^{H}(u)$ are the unique solutions to $\Phi_1(s_1)=u$, where $\zeta_1^L(u)<\zeta_1^H(u)$. Symmetrically, we define the inverse branches of $\Phi_2$ by
\begin{align*}
    \zeta_2^L&:\Phi_2([0,\tilde{S}_2])\to [0,\tilde{S}_2],\\
    \zeta_2^H&:\Phi_2([\tilde{S}_2,m_2])\to [\tilde{S}_2,m_2],
\end{align*}
where $\zeta_2^L(u)$ and $\zeta_2^H(u)$ are the unique solutions to $\Phi_2(s_2)=u$, where $\zeta_2^L(u)<\zeta_2^H(u)$.

Using these inverse branches, for any $s_2$ such that $\Phi_2(s_2)$ lies in the range of $\Phi_1$, we define the two indifference branches by
\begin{align*}
    g_L(s_2)&:=\zeta_1^L(\Phi_2(s_2)),\\
    g_H(s_2)&:=\zeta_1^H(\Phi_2(s_2)).
\end{align*}
Then, by construction, $\Phi_1(g_L(s_2))=\Phi_2(s_2)$, $\Phi_1(g_H(s_2))=\Phi_2(s_2)$, and whenever both branches exist we have $g_L(s_2)<\tilde{S}_1<g_H(s_2)$. We now extend these results to the full switching curves.

First, fix $s_2\in [0,\tilde{S}_2)$ and consider an initial state $s_1\in[\tilde{S}_1,m_1]$. Suppose $(\mathcal{A}_1,\mathcal{A}_2)=(0,A)$ (i.e., all traffic is allocated to creator 2), until $\mathcal{S}_2$ reaches the level
\begin{align*}
    \hat{s}_2(s_1):=\zeta_2^H(\Phi_1(s_1)).
\end{align*}
At this time the state lies on the upper indifference curve branch $s_1=g_H(s_2)$, after which the policy switches to an allocation such that $\Phi_1=\Phi_2$ is maintained henceforward.

Under this policy, $\mathcal{S}_1(t)\equiv s_1$, $\mathcal{S}_2(t)=\Xi_2(t;s_2)$ for $t\in[0,T_2(s_2,\hat{s}_2(s_1)]$, and $\Phi_1=\Phi_2$ afterwards. Therefore, by \eqref{eq:E_integral_representation}, $I(0)=0$ is equivalent to
\begin{align}
    0=\delta\int_0^{T_2(s_2,\hat{s}_2(s_1))}e^{-\delta t}\left[\Phi_1(s_1)-\Phi_2(\Xi_2(t;s_2))\right]\,dt.\label{eq:barE_eq_0_lower_right}
\end{align}
We define the upper switching curve $g_1(s_2)$ for $s_2\in[0,\tilde{S}_2)$ as the unique solution $s_1\in[\tilde{S}_1,m_1]$ to equation \eqref{eq:barE_eq_0_lower_right} that satisfies $\Phi_1(s_1)>\Phi_2(s_2)$.

Then, fix $s_2\in(\tilde{S}_2,m_2]$ and $s_1\in[0,\tilde{S}_1]$. Suppose $(\mathcal{A}_1,\mathcal{A}_2)=(A,0)$ (i.e., all traffic is allocated to creator 1), until $\mathcal{S}_1$ reaches $g_H(s_2)$. Afterwards, the policy switches to an allocation such that $\Phi_1=\Phi_2$ is maintained henceforward.

Under this policy, $\mathcal{S}_2(t)\equiv s_2$, $\mathcal{S}_1(t)=\Xi_1(t;s_1)$ for $t\in[0,T_1(s_1,g_H(s_2))]$, and $\Phi_1=\Phi_2$ afterwards. Thus by \eqref{eq:E_integral_representation}, $I(0)=0$ is equivalent to
\begin{align}
    0=\delta\int_0^{T_1(s_1,g_H(s_2))}e^{-\delta t}\left[\Phi_1(\Xi_1(t;s_1))-\Phi_2(s_2)\right]\,dt.\label{eq:barE_eq_0_upper_left}
\end{align}
We define the lower switching curve $g_2(s_2)$ for $s_2\in(\tilde{S}_2,m_2]$ as the unique solution to $s_1\in[0,\tilde{S}_1]$ to equation \eqref{eq:barE_eq_0_upper_left} that satisfies $\Phi_1(s_1)<\Phi_2(s_2)$.

Combining above, we define
\begin{align}
    g_1(s_2)&=\begin{cases}
        \text{the solution to }\eqref{eq:barE_eq_0_lower_right}&\text{if }s_2\in[0,\tilde{S}_2),\\
        g_H(s_2)&\text{if }s_2\in[\tilde{S}_2,m_2],
    \end{cases}\\
    g_2(s_2)&=\begin{cases}
        g_L(s_2)&\text{if }s_2\in[0,\tilde{S}_2],\\
        \text{the solution to }\eqref{eq:barE_eq_0_upper_left}&\text{if }s_2\in(\tilde{S}_2,m_2].
    \end{cases}\label{eq:def_g1_g2}
\end{align}
On the set $\{\mathcal{S}_1(t)=g_1(\mathcal{S}_2(t)), \mathcal{S}_2(t)\geq\tilde{S}_2\}$, we have by construction $\Phi_1(\mathcal{S}_1(t))=\Phi_2(\mathcal{S}_2(t))$. For states with $\mathcal{S}_2(t)>\tilde{S}_2$, we allocate traffic so that this equality is preserved onward, i.e., we impose $\dot{\Phi}_1(\mathcal{S}_1(t))=\dot{\Phi}_2(\mathcal{S}_2(t))$, and set $(\mathcal{A}_1(t),\mathcal{A}_2(t))=(\bar h(\mathcal{S}_2(t)),A-\bar h(\mathcal{S}_2(t)))$, where
\begin{align}
    \bar{h}(\mathcal{S}_2(t))
=A\cdot\frac{\rho_2\,\mathcal{F}_{B,2}'(\mathcal{S}_2(t))\,\mathcal{F}_{B,2}(\mathcal{S}_2(t))}
{\rho_1\,\mathcal{F}_{B,1}'(g_1(\mathcal{S}_2(t)))\,\mathcal{F}_{B,1}(g_1(\mathcal{S}_2(t)))
+\rho_2\,\mathcal{F}_{B,2}'(\mathcal{S}_2(t))\,\mathcal{F}_{B,2}(\mathcal{S}_2(t))}.\label{eq:def_h}
\end{align}
With $g_1(\cdot)$, $g_2(\cdot)$, and $\bar h(\cdot)$ defined above, the proposed policy in
Theorem \ref{Thm:Opt_policy_2Crtr_Bs} is
\begin{align}\label{eq:proposed_policy}
(\mathcal{A}_1^*(t),\mathcal{A}_2^*(t))=
\begin{cases}
(A,0),& \text{if } g_2(\mathcal{S}_2(t))\le \mathcal{S}_1(t)<g_1(\mathcal{S}_2(t)),\\
(0,A),& \text{if } \mathcal{S}_1(t)<g_2(\mathcal{S}_2(t)) \text{ or } \mathcal{S}_1(t)>g_1(\mathcal{S}_2(t))\\
& \text{or } \mathcal{S}_1(t)=g_1(\mathcal{S}_2(t)),\ \mathcal{S}_2(t)\le \tilde{\mathcal{S}}_2,\\
(\bar h(\mathcal{S}_2(t)),A-\bar h(\mathcal{S}_2(t))),&
\text{if } \mathcal{S}_1(t)=g_1(\mathcal{S}_2(t)),\ \mathcal{S}_2(t)> \tilde{\mathcal{S}}_2.
\end{cases}
\end{align}
where $\tilde{\mathcal{S}}_2:=m_2(q-p\rho_2)/(2q)$.

We summarize several properties for the policy in the following lemma.
\begin{lemma}\label{lem:proposed_policy_properties}
    By construction of $g_1(\cdot)$ and $g_2(\cdot)$,
    \begin{enumerate}
        \item If $\mathcal{S}_1(t)=g_1(\mathcal{S}_2(t))$, and $\mathcal{S}_2(t)\geq\tilde{S}_2$, then $\Phi_1(\mathcal{S}_1(t))=\Phi_2(\mathcal{S}_2(t))$;
        \item If $\mathcal{S}_1(t)=g_2(\mathcal{S}_2(t))$, and $\mathcal{S}_2(t)\leq\tilde{S}_2$, then $\Phi_1(\mathcal{S}_1(t))=\Phi_2(\mathcal{S}_2(t))$;
        \item If $\mathcal{S}_1(t)=g_1(\mathcal{S}_2(t))$, $\mathcal{S}_2(t)<\tilde{S}_2$, then $\Phi_1(\mathcal{S}_1(t))>\Phi_2(\mathcal{S}_2(t))$ and $I(t)=0$;
        \item If $\mathcal{S}_1(t)=g_2(\mathcal{S}_2(t))$, $\mathcal{S}_2(t)>\tilde{S}_2$, then $\Phi_1(\mathcal{S}_1(t))<\Phi_2(\mathcal{S}_2(t))$ and $I(t)=0$.
    \end{enumerate}
\end{lemma}
\textit{Proof. }The results are directly from construction of $g_1(\cdot)$, $g_2(\cdot)$, as well as the identity \eqref{eq:E_integral_representation}, thus omitted.\qed

To establish the admissibility of our proposed policy, note that the allocations $(0,A)$ and $(A,0)$ are clearly feasible. Therefore, it suffices to verify that the interior allocation $(\bar{h}(\cdot),A-\bar{h}(\cdot))$ also satisfies the feasibility constraints, which is shown in the following lemma.
\begin{lemma}\label{lem:final_seg_admissible}
    When $\mathcal{S}_1(t)=g_1(\mathcal{S}_2(t))$, $\mathcal{S}_2(t)>\tilde{S}_2$, the traffic allocation $(\bar{h}(\cdot),A-\bar{h}(\cdot))$ such that $\dot{\Phi}_1(\mathcal{S}_1(t))=\dot{\Phi}_2(\mathcal{S}_2(t))$ is admissible, i.e., $\bar{h}(\cdot)\in[0,A]$.
\end{lemma}
\textit{Proof. }Differentiating $\Phi_i(\mathcal{S}_1(t))$ yields
\begin{align*}
    \dot{\Phi}_i(\mathcal{S}_1(t))&=\rho_i\cdot\frac{r}{\delta}\frac{\partial \mathcal{F}_{B,i}(\mathcal{S}_i(t))}{\partial\mathcal{S}_i(t)}\dot{\mathcal{S}}_i(t),\\
    &=\rho_i\cdot\frac{r}{\delta}\frac{\partial \mathcal{F}_{B,i}(\mathcal{S}_i(t))}{\partial\mathcal{S}_i(t)}\mathcal{A}_i(t)\mathcal{F}_{B,i}(t).
\end{align*}
Imposing $\dot{\Phi}_1(\mathcal{S}_1(t))=\dot{\Phi}_2(\mathcal{S}_2(t))$ and apply $\mathcal{A}_2(t)=A-\mathcal{A}_1(t)$ yields 
\begin{align}
    \mathcal{A}_1(t)=A\cdot\frac{\rho_2\frac{\partial\mathcal{F}_{B,2}(\mathcal{S}_2(t))}{\partial\mathcal{S}_2(t)}\mathcal{F}_{B,2}(t)}{\rho_1\frac{\partial\mathcal{F}_{B,1}(\mathcal{S}_1(t))}{\partial\mathcal{S}_1(t)}\mathcal{F}_{B,1}(t)+\rho_2\frac{\partial\mathcal{F}_{B,2}(\mathcal{S}_2(t))}{\partial\mathcal{S}_2(t)}\mathcal{F}_{B,2}(t)}.\label{Eq:Pf_Opt_policy_2Crtr_Bs_Case1_A1t}
\end{align}
Note that $\mathcal{S}_1(t)=g_1(\mathcal{S}_2(t))$ directly implies $\mathcal{S}_1(t)\geq\tilde{S}_1$. Since $\mathcal{S}_i(t)$ is non-decreasing and that $\mathcal{S}_i(t)\geq\tilde{S}_i=m_i(q-p\rho_i)/(2q)$, consider
\begin{align*}
    \frac{\partial\mathcal{F}_{B,i}(\mathcal{S}_i(t))}{\partial\mathcal{S}_i(t)}=\frac{q-p\rho_i}{m_i}-\frac{2q}{m_i^2}\mathcal{S}_i(t)
    =\frac{1}{m_i}\left(q-p\rho_i-\frac{2q}{m_i}\mathcal{S}_i(t)\right)\leq 0.
\end{align*}
Also, $\mathcal{F}_{B,i}(t)\geq 0$ since $\mathcal{S}_i(t)\in[0,m_i]$. Thus,
\begin{align*}
    \rho_i\frac{\partial\mathcal{F}_{B,i}(\mathcal{S}_i(t))}{\partial\mathcal{S}_i(t)}\mathcal{F}_{B,i}(t)\leq 0\text{ for }i=1,2.
\end{align*}
Hence by \eqref{Eq:Pf_Opt_policy_2Crtr_Bs_Case1_A1t}, $\mathcal{A}_1(t)\in[0,A]$ is admissible.\qed

\subsection{Verification of Optimality}
We prove Theorem \ref{Thm:Opt_policy_2Crtr_Bs} by verifying the sufficient conditions in Lemma \ref{lem:PMP_suff_conds}. We first formulate the current-value Hamiltonian and derive the costate dynamics, which lead to a switching function $E(t)$ with property given in \ref{lem:Et_dynamics}. Imposing the transversality condition, we obtain an integral representation of $E(t)$ and show the Hamiltonian is concave in the state variables. We then verify the PMP maximization condition by tracking the sign of $E(t)$ across the regions defined by the switching curves $g_1(\cdot)$ and $g_2(\cdot)$, and characterizing the behavior of $E(t)$ on the boundary set where $E(t)=0$.

The current-value Hamiltonian associated to the problem is
\begin{align}
    &\mathcal{H}(\mathcal{S}_1(t),\mathcal{S}_2(t),\mathcal{A}_1(t),\mathcal{A}_2(t),\lambda_1(t),\lambda_2(t),t)\nonumber\\
    &=\rho_1(\mathcal{A}_1(t)+r\mathcal{S}_1(t))+\rho_2(\mathcal{A}_2(t)+r\mathcal{S}_2(t))\nonumber\\
    &+\lambda_1(t)\mathcal{A}_1(t)\left(p\rho_1+q\cdot\frac{\mathcal{S}_1(t)}{m_1}\right)\left(1-\frac{\mathcal{S}_1(t)}{m_1}\right)\nonumber\\
    &+\lambda_2(t)\mathcal{A}_2(t)\left(p\rho_2+q\cdot\frac{\mathcal{S}_2(t)}{m_2}\right)\left(1-\frac{\mathcal{S}_2(t)}{m_2}\right)\label{eq:Ham_2Crtr_Bass}
\end{align}
Of which, $\lambda_1(t),\lambda_2(t)$ are the current-value costate variables that evolve according to
\begin{align}
    \dot{\lambda}_i(t)=\delta\lambda_i(t)-\frac{\partial\mathcal{H}}{\partial\mathcal{S}_i(t)}=\left[\delta-\mathcal{A}_i(t)\left(\frac{q-p\rho_i}{m_i}-\frac{2q}{m_i^2}\mathcal{S}_i(t)\right)\right]\lambda_i(t)-\rho_ir\text{ for }i=1,2.\label{eq:lambda_i_evolution}
\end{align}
Since $\mathcal{A}_1(t)+\mathcal{A}_2(t)=A$ for all $t$, we can rewrite \eqref{eq:Ham_2Crtr_Bass} as
\begin{align*}
&\Bigg[
\rho_1-\rho_2
+\lambda_1(t)\left(p\rho_1+q\cdot\frac{\mathcal{S}_1(t)}{m_1}\right)\left(1-\frac{\mathcal{S}_1(t)}{m_1}\right)
-\lambda_2(t)\left(p\rho_2+q\cdot\frac{\mathcal{S}_2(t)}{m_2}\right)\left(1-\frac{\mathcal{S}_2(t)}{m_2}\right)
\Bigg]\mathcal{A}_1(t)\\
&+
\Bigg[
\rho_2 A
+\rho_1 r\mathcal{S}_1(t)
+\rho_2 r\mathcal{S}_2(t)
+\lambda_2(t)A\left(p\rho_2+q\cdot\frac{\mathcal{S}_2(t)}{m_2}\right)\left(1-\frac{\mathcal{S}_2(t)}{m_2}\right)
\Bigg]\\
&:=E(t)\mathcal{A}_1(t)+C(t),
\end{align*}
where $C(t)$ is some term independent of $\mathcal{A}_1(t)$. We define
\begin{align*}
    \mathcal{F}_{B,i}(\mathcal{S}_i(t)):=\left(p\rho_i+q\cdot\frac{\mathcal{S}_i(t)}{m_i}\right)\left(1-\frac{\mathcal{S}_i(t)}{m_i}\right),
\end{align*}
then $E(t)$ can be written as
\begin{align*}
    \rho_1-\rho_2+\lambda_1(t)\mathcal{F}_{B,1}(\mathcal{S}_1(t))-\lambda_2(t)\mathcal{F}_{B,2}(\mathcal{S}_2(t)).
\end{align*}
It turns out that we can characterize the evolution of $E(t)$ using the marginal value function $\Phi_i(\mathcal{S}_i(t))$ for each creator $i$, established in the following lemma.
\begin{lemma}\label{lem:Et_dynamics}
    For any admissible policy $\mathcal{A}_i(t)$ and corresponding state trajectories $\mathcal{S}_i(t)$, $i=1,2$, if the costate variables $\lambda_i(t)$ satisfies \eqref{eq:lambda_i_evolution}, then
    \begin{align*}
        \dot{E}(t)=\delta E(t)-\delta\left[\Phi_1(\mathcal{S}_1(t))-\Phi_2(\mathcal{S}_2(t))\right].
    \end{align*}
\end{lemma}
\textit{Proof.} Recall that
\begin{align*}
    E(t)=\rho_1-\rho_2+\lambda_1(t)\mathcal{F}_{B,1}(\mathcal{S}_1(t))-\lambda_2(t)\mathcal{F}_{B,2}(\mathcal{S}_2(t)),
\end{align*}
differentiate $E(t)$ we obtain
\begin{align*}
    \dot{E}(t)=\dot{\lambda}_1(t)\mathcal{F}_{B,1}(\mathcal{S}_1(t))+\lambda_1(t)\dot{\mathcal{F}}_{B,1}(\mathcal{S}_1(t))-\dot{\lambda}_2(t)\mathcal{F}_{B,2}(\mathcal{S}_2(t))-\lambda_2(t)\dot{\mathcal{F}}_{B,2}(\mathcal{S}_2(t)).
\end{align*}
Note that $\dot{\mathcal{F}}_{B,i}(\mathcal{S}_i(t))=\mathcal{F}_{B,i}'(\mathcal{S}_i(t))\cdot\dot{\mathcal{S}}_i(t)=\mathcal{F}_{B,i}'(\mathcal{S}_i(t))\cdot \mathcal{A}_i(t)\mathcal{F}_{B,i}(\mathcal{S}_i(t))$. By \eqref{eq:lambda_i_evolution},
\begin{align*}
    \dot{\lambda}_i(t)=\delta\lambda_i(t)-\rho_i r-\lambda_i(t)\mathcal{A}_i(t)\mathcal{F}_{B,i}'(\mathcal{S}_i(t)).
\end{align*}
Therefore,
\begin{align*}
    \dot{E}(t)&=\sum_{i=1}^2(-1)^{i+1}\Big[\delta\lambda_i(t)\mathcal{F}_{B,i}(\mathcal{S}_i(t))-\rho_i r\mathcal{F}_{B,i}(\mathcal{S}_i(t))\\
    &-\lambda_i(t)\mathcal{A}_i(t)\mathcal{F}_{B,i}'(\mathcal{S}_i(t))\mathcal{F}_{B,i}(\mathcal{S}_i(t))+\lambda_i(t)\mathcal{F}_{B,i}'(\mathcal{S}_i(t))\mathcal{A}_i(t)\mathcal{F}_{B,i}(\mathcal{S}_i(t))\Big]\\
    &=\sum_{i=1}^2(-1)^{i+1}(\delta\lambda_i(t)-\rho_ir)\mathcal{F}_{B,i}(\mathcal{S}_i(t))\\
    &=\delta (E(t)-(\rho_1-\rho_2))-r(\rho_1\mathcal{F}_{B,1}(\mathcal{S}_1(t))-\rho_2\mathcal{F}_{B,2}(\mathcal{S}_2(t))).
\end{align*}
Using the definition of $\Phi_i(\cdot)$ in \eqref{eq:phi_i}, we obtain
\begin{align*}
    \delta\left[\Phi_1(\mathcal{S}_1(t))-\Phi_2(\mathcal{S}_2(t))\right]=\delta(\rho_1-\rho_2)+r\left[\rho_1\mathcal{F}_{B,1}(\mathcal{S}_1(t))-\rho_2\mathcal{F}_{B,2}(\mathcal{S}_2(t))\right].
\end{align*}
Hence,
\begin{align}
    \dot{E}(t)=\delta E(t)-\delta\left[\Phi_1(\mathcal{S}_1(t))-\Phi_2(\mathcal{S}_2(t))\right].\label{eq:dot_Et}
\end{align}\qed

Next, we establish Lemma \ref{lem:lambda_pos}, which shows that under an additional transversality condition on the costate variables, the costates remain strictly positive under any admissible policy.
\begin{lemma}\label{lem:lambda_pos}
    For any admissible policy $\mathcal{A}_i(t)$ and corresponding state trajectories $\mathcal{S}_i(t)$, $i=1,2$, if the costate variables $\lambda_i(t)$ satisfy \eqref{eq:lambda_i_evolution} and the transversality condition $\lim_{t\to\infty}e^{-\delta t}\lambda_i(t)=0$, then $\lambda_i(t)>0$ for all $t>0$.
\end{lemma}
\textit{Proof. }Recall that
\begin{align*}
    \dot{\lambda}_i(t)=\left[\delta-\mathcal{A}_i(t)\left(\frac{q-p\rho_i}{m_i}-\frac{2q}{m_i^2}\mathcal{S}_i(t)\right)\right]\lambda_i(t)-\rho_ir.
\end{align*}
Denote
\begin{align*}
    a_i(t)&:=\delta-\mathcal{A}_i(t)\left(\frac{q-p\rho_i}{m_i}-\frac{2q}{m_i^2}\mathcal{S}_i(t)\right),\\
    M_i(t)&:=\exp\left(-\int_0^t a_i(u)\,du\right)>0.
\end{align*}
So $\dot{\lambda}_i(t)=a_i(t)\lambda_i(t)-\rho_i r$,
\begin{align*}
    \frac{d}{dt}(\lambda_i(t)M_i(t))&=\left[a_i(t)\lambda_i(t)-\rho_i r\right]M_i(t)+\lambda_i(t)(-a_i(t)M_i(t))\\
    &=a_i(t)\lambda_i(t)M_i(t)-\rho_i rM_i(t)-a_i(t)\lambda_i(t)M_i(t)\\
    &=-\rho_i rM_i(t).
\end{align*}
Therefore, we can express $\lambda_i(t)$ via
\begin{align*}
    \lambda_i(T)M_i(T)-\lambda_i(t)M_i(t)&=-\rho_i r\int_t^T M_i(s)\,ds\\
    \lambda_i(t)&=\frac{\lambda_i(T)M_i(T)}{M_i(t)}+\frac{\rho_i r}{M_i(t)}\int_t^T M_i(s)\,ds.
\end{align*}
We next show that $e^{\delta T}M_i(T)$ is bounded above, consider
\begin{align*}
    M_i(T)&=\exp\left(-\int_0^T a_i(u)\,du\right)\\
    &=\exp\left(-\int_0^T \left[\delta-\mathcal{A}_i(t)\mathcal{F}_{B,i}'(\mathcal{S}_i(t))\right]\,dt\right)\\
    &=e^{-\delta T}\exp\left(\int_0^T \mathcal{A}_i(t)\mathcal{F}_{B,i}'(\mathcal{S}_i(t))\,dt\right).
\end{align*}
For all finite $t$, $\mathcal{S}_i(t)\in[0,m_i)$ and $\mathcal{F}_{B,i}(\mathcal{S}_i(t))>0$, so $\log\mathcal{F}_{B,i}(\mathcal{S}_i(t))$ is well defined, we then differentiate $\log\mathcal{F}_{B,i}(\mathcal{S}_i(t))$ to obtain
\begin{align*}
    \frac{d}{dt}\log\mathcal{F}_{B,i}(\mathcal{S}_i(t))=\frac{\mathcal{F}_{B,i}'(\mathcal{S}_i(t))}{\mathcal{F}_{B,i}(\mathcal{S}_i(t))}\mathcal{A}_i(t)\mathcal{F}_{B,i}(\mathcal{S}_i(t))=\mathcal{F}_{B,i}'(\mathcal{S}_i(t))\mathcal{A}_i(t).
\end{align*}
Hence,
\begin{align*}
    M_i(T)&=e^{-\delta T}\exp\left(\int_0^T \frac{d}{dt}\log \mathcal{F}_{B,i}(\mathcal{S}_i(t))\,dt\right)\\
    &=e^{-\delta T}\frac{\mathcal{F}_{B,i}(\mathcal{S}_i(T))}{\mathcal{F}_{B,i}(\mathcal{S}_i(0))}.
\end{align*}
Since $\mathcal{S}_i(t)$ is bounded above, and $M_i(T)>0$, $\frac{\mathcal{F}_{B,i}(\mathcal{S}_i(T))}{\mathcal{F}_{B,i}(\mathcal{S}_i(0))}$ is also bounded above, therefore, there exists $C_i>0$ such that $e^{\delta T}M_i(T)\leq C_i$.

By the transversality condition, and take $T\to\infty$, we have
\begin{align*}
    \lim_{T\to\infty}\lambda_i(T)M_i(T)=\lim_{T\to\infty}(e^{-\delta T}\lambda_i(T))(e^{\delta T}M_i(T))=0.
\end{align*}
So
\begin{align*}
    \lambda_i(t)=\frac{\rho_i r}{M_i(t)}\int_t^{\infty}M_i(s)\,ds>0.
\end{align*}
\qed

We now state a sufficient condition that we will use to verify the optimality of our proposed policy. 
\begin{lemma}\label{lem:PMP_suff_conds}
    If there exist costate variables $\lambda_i(t)$ (for $i=1,2$) such that 1) the transversality condition $\lim_{t\to\infty}e^{-\delta t}\lambda_i(t)=0$ is satisfied; 2) the Hamiltonian $\mathcal{H}$ is concave in $\mathcal{S}$ at each $t$; 3) $\mathcal{A}_1^*(t),\mathcal{A}_2^*(t)$ are admissible, so $\mathcal{A}_2^*(t)=A-\mathcal{A}_1^*(t)$, and $\mathcal{A}_1^*(t)$ satisfies
    \begin{align}
        \mathcal{A}_1^*(t)=\begin{cases}
            A&\text{if }E(t)>0,\\
            0&\text{if }E(t)<0,\\
            \text{any value in }[0,A]&\text{if }E(t)=0,
        \end{cases}\label{eq:Bs_PMP}
    \end{align}then $\mathcal{A}_i^*(t)$ and associated state trajectory $\mathcal{S}_i^*(t)$ are optimal.
\end{lemma}
\textit{Proof. }Since the problem is formulated over an infinite horizon with a free terminal state, by \cite{sethi2021optimal} (3.99), the sufficiency condition can be extended to infinite horizon by including the transversality condition $\lim_{t\to\infty}e^{-\delta t}\lambda_i(t)=0$ for $i=1,2$, while the concavity condition and the PMP necessary condition \eqref{eq:Bs_PMP} are still required as in finite horizon case.\qed

We now proceed to verify that the proposed policy \eqref{eq:proposed_policy} satisfies the three conditions in the preceding sufficiency lemma. Fix costate variables $\lambda_i(t)$, $i=1,2$ that satisfy the costate dynamics \eqref{eq:lambda_i_evolution} and the transversality condition
\begin{align*}
    \lim_{t\to\infty}e^{-\delta t}\lambda_i(t)=0,\quad i=1,2.
\end{align*}
For such costates, the boundedness of $\mathcal{F}_{B,i}(\mathcal{S}_i(t))$ implies $\lim_{t\to\infty}e^{-\delta t}E(t)=0$. By \eqref{eq:dot_Et}, the switching function $E(t)$ coincides with $I(t)$, and admits the representation \eqref{eq:E_integral_representation},
\begin{align*}
    E(t)=\delta\int_t^{\infty}e^{\delta(t-s)}\left[\Phi_1(\mathcal{S}_1(s))-\Phi_2(\mathcal{S}_2(s))\right]\,ds.
\end{align*}
The admissibility of $\mathcal{A}^*(t)$ has already been established by Lemma \ref{lem:final_seg_admissible}. It remains to verify (i) concavity of the Hamiltonian in $\mathcal{S}$ and (ii) the PMP maximization condition \eqref{eq:Bs_PMP}.

Firstly, we verify the concavity condition. Since $\mathcal{S}_i$ enters $\mathcal{H}$ separately across $i$, it suffices to show $\partial^2\mathcal{H}/\partial \mathcal{S}_i(t)^2\le 0$ for each $i$. Note that
\begin{align*}
    \frac{d^2}{d\mathcal{S}_i^2}\mathcal{F}_{B,i}(\mathcal{S}_i)=-\frac{2q}{m_i^2}<0.
\end{align*}
Under any admissible policy, $\mathcal{A}_i(t)\in[0,A]$. Moreover, by Lemma \ref{lem:lambda_pos}, the transversality condition implies $\lambda_i(t)>0$ for all $t>0$. Hence,
\begin{align*}
    \frac{\partial^2\mathcal{H}}{\partial\mathcal{S}_i(t)^2}=\lambda_i(t)\mathcal{A}_i(t)\cdot\left(-\frac{2q}{m_i^2}\right)\leq 0.
\end{align*}
so $\mathcal{H}$ is concave in $\mathcal{S}$ at each $t$.

It remains to verify the PMP maximization condition \eqref{eq:Bs_PMP}. Under the transversality condition, we work with $E(t)$ and define
\begin{align*}
    D(t):=\Phi_1(\mathcal{S}_1(t))-\Phi_2(\mathcal{S}_2(t)).
\end{align*}
Then we have
\begin{align*}
    \dot{E}(t)=\delta(E(t)-D(t)),\quad E(t)=\delta\int_t^{\infty}e^{\delta(t-s)}D(s)\,ds.
\end{align*}
We verify \eqref{eq:Bs_PMP} by tracking the sign of $E(t)$ across the regions determined by the switching curves $g_1(\cdot)$, $g_2(\cdot)$. Define the three sets
\begin{align*}
    \mathcal{R}_+&:=\{(s_1,s_2): g_2(s_2)\leq s_1<g_1(s_2)\},\\
    \mathcal{R}_-&:=\{(s_1,s_2): s_1<g_2(s_2)\text{ or }s_1>g_1(s_2)\text{ or }s_1=g_1(s_2),s_2\leq\tilde{S}_2\},\\
\mathcal{M}&:=\{(s_1,s_2): s_1=g_1(s_2),s_2> \tilde{S}_2\}.
\end{align*}
The candidate policy chooses $(A,0)$ on $\mathcal R_+$, $(0,A)$ on $\mathcal R_-$, and $(\bar h(s_2),A-\bar h(s_2))$ on $\mathcal M$. First, if the trajectory is on $\mathcal M$, then by Lemma \ref{lem:proposed_policy_properties}(1) we have $D(t)=0$. Moreover, $\bar h$ is defined so that $\dot D(t)=0$, hence the trajectory remains on $\mathcal M$ and $D(t)\equiv 0$ thereafter. By the integral representation of $E(t)$, this implies $E(t)\equiv 0$ on $\mathcal M$. Therefore the PMP condition allows any feasible control there, including $(\bar h(s_2),A-\bar h(s_2))$.

Next, consider the boundary $\{ \mathcal S_1(t)=g_1(\mathcal S_2(t)),\ \mathcal S_2(t)<\tilde S_2\}$. By Lemma \ref{lem:proposed_policy_properties}(3), we have $E(t)=0$ and $D(t)>0$, so
\[
\dot E(t)= -\delta D(t)<0.
\]
Hence $E(t)$ crosses zero with negative slope, which matches the switch from $(A,0)$ to $(0,A)$ at this boundary.

Similarly, on the boundary $\{ \mathcal S_1(t)=g_2(\mathcal S_2(t)),\ \mathcal S_2(t)>\tilde S_2\}$, Lemma \ref{lem:proposed_policy_properties}(4) gives $E(t)=0$ and $D(t)<0$, so
\[
\dot E(t)= -\delta D(t)>0.
\]
Hence $E(t)$ crosses zero with positive slope, which matches the switch from $(0,A)$ to $(A,0)$ at this boundary.

Finally, $E(t)$ is continuous, so its sign can change only when it hits zero. By construction, the candidate trajectory can hit $E(t)=0$ only on the three sets above: the balanced boundary $\mathcal M$, the upper switching boundary with $\mathcal S_2<\tilde S_2$, and the lower switching boundary with $\mathcal S_2>\tilde S_2$. On $\mathcal M$, $E(t)$ remains equal to zero; on the other two boundaries, the sign change direction is determined by the sign of $\dot E(t)$ as shown above. Therefore, between consecutive hitting times of these sets, the sign of $E(t)$ is constant. On $\mathcal R_+$, the candidate policy chooses $(A,0)$ consistently with the case $E(t)>0$; on $\mathcal R_-$, it chooses $(0,A)$ consistently with the case $E(t)<0$ in the interior; and on the switching sets where $E(t)=0$, it selects a feasible control. Hence the candidate policy satisfies the PMP maximization condition \eqref{eq:Bs_PMP}.

In summary, the proposed policy satisfies three conditions in Lemma \ref{lem:PMP_suff_conds}, we complete the verification of optimality.

\section{Proof of Proposition \ref{Prop:2Crtr_Long_base}}
We prove the linear and Bass growth cases separately. 

\subsection{Linear Growth}
Under linear growth, recall that by Proposition \ref{Prop:Opt_policy_2Crtr_Ln_Psi_Eq_gS1}, the switching line $\mathcal{S}_2=g(\mathcal{S}_1)$ is equivalent to $\Psi_1=\Psi_2$, and the optimal policy allocates all traffic to the creator with the larger $\Psi_i$, while on the switching line it allocates traffic so as to maintain $\Psi_1=\Psi_2$. 

We now begin to prove Proposition \ref{Prop:2Crtr_Long_base} Part (b) under linear growth. We first characterize when creator 2 ever receives positive traffic. If creator 2 is inactive, then $\mathcal{S}_2\equiv S_2$, so $\Psi_2(t)\equiv \Psi_2(0)$. Under full allocation to creator 1, $\Psi_1(t)$ decreases continuously to $\lim_{t\to\infty}\Psi_1(t)=\rho_1$. Hence, creator 2 is eventually activated if and only if its constant marginal value while inactive is strictly above this lower limit, i.e., $\Psi_2(0)>\rho_1$.

Indeed, if $\Psi_2(0)\leq\rho_1$, then $\Psi_1(t)\geq\rho_2\geq \Psi_2(0)$ for all $t$, so the switching condition $\Psi_1=\Psi_2$ is never reached and creator 2 never receives traffic. Conversely, if $\Psi_2(0)>\rho_1$, then either $\Psi_2(0)\geq\Psi_1(0)$, in which case creator 2 is served immediately, or $\Psi_2(0)<\Psi_1(0)$, in which case creator 1 is served first and $\Psi_1(t)\downarrow \rho_1<\Psi_2(0)$, so by continuity there is a finite time at which $\Psi_1=\Psi_2$, after which creator 2 receives positive traffic.

Therefore, the exact threshold is characterized by $\Psi_2(0)=\rho_1$, namely
\begin{align*}
    \rho_2\left[1+\frac{pr}{\delta}\cdot \rho_2\left(1-\frac{S_2}{m_2}\right)\right]=\rho_1,
\end{align*}
solving the quadratic equation yields
\begin{align*}
    \rho_2^{\,\min}=\rho_{2,\text{linear}}^{\,\min}=\frac{-1+\sqrt{1+4\rho_1\cdot\frac{pr}{\delta}\left(1-\frac{S_2}{m_2}\right)}}{2\cdot\frac{pr}{\delta}\left(1-\frac{S_2}{m_2}\right)}.
\end{align*}
To prove the monotonicity in $p$ and $r$, define
\begin{align*}
    a:=\frac{pr}{\delta}\left(1-\frac{S_2}{m_2}\right)>0,
\end{align*}
so that $\rho_{2,\text{linear}}^{\,\min}$ is the positive solution of $a\rho^2+\rho-\rho_1=0$. Differentiating it gives
\begin{align*}
    \frac{\partial\rho}{\partial a}=-\frac{\rho^2}{2a\rho+1}<0.
\end{align*}
Since $a$ is increasing in both $p$ and $r$, it follows that $\rho_{2,\text{linear}}^{\,\min}$ is decreasing in $p$ and $r$.

Then, we proceed to prove Part (a) under linear growth. If $\rho_2\leq \rho_2^{\min}$, then by Part (b) creator 2 never receives traffic, so $\mathcal{A}_1^*(t)\equiv A$, and therefore $\mathcal{S}_1(t)\uparrow m_1$.

Now suppose $\rho_2>\rho_2^{\min}$. Then creator 2 is activated in finite time. By Theorem \ref{Thm:Opt_policy_2Crtr_Ln}, after a finite initial phase the system reaches the switching line $\mathcal{S}_2=g(\mathcal{S}_1)$, and thereafter the optimal allocation is $\mathcal{A}_1^*(t)=h(\mathcal{S}_1(t))$.

From the explicit formula of $h(\cdot)$ in Theorem \ref{Thm:Opt_policy_2Crtr_Ln}, we have $0<h(\mathcal{S}_1)<A$ whenever $\mathcal{S}_1<m_1$ and $\mathcal{S}_2<m_2$. Hence, on the balanced phase,
\begin{align*}
    \dot{\mathcal{S}}_1(t)
    =
    h(\mathcal{S}_1(t))\,p\rho_1\left(1-\frac{\mathcal{S}_1(t)}{m_1}\right)>0
\end{align*}
as long as $\mathcal{S}_1(t)<m_1$. Since $\mathcal{S}_1(t)$ is increasing and bounded above by $m_1$, it has a limit. This limit cannot be strictly below $m_1$, because then the right-hand side above would remain strictly positive near the limit. Therefore $\lim_{t\to\infty}\mathcal{S}_1(t)=m_1$.

Finally, we prove Part (c) under linear growth. Assume $\rho_2>\rho_2^{\min}$, so creator 2 receives positive traffic. Since $\Psi_2(\cdot)$ is strictly decreasing in $\mathcal{S}_2$, there is a unique $\bar{S}_2\in(0,m_2)$ satisfying $\Psi_2(\bar{S}_2)=\rho_1$. Using the formula of $\Psi_2$, this gives
\begin{align*}
    \rho_2\left[1+\frac{pr}{\delta}\rho_2\left(1-\frac{\bar{S}_2}{m_2}\right)\right]=\rho_1,
\end{align*}
hence
\begin{align*}
    \bar{S}_2=\bar{S}_{2,\mathrm{linear}}=
    m_2\left(1-\frac{\rho_1-\rho_2}{\rho_2^2}\cdot\frac{\delta}{pr}\right).
\end{align*}
We next show that $\mathcal{S}_2(t)\leq \bar{S}_2$ for all $t\geq 0$. Since $\rho_2>\rho_2^{\min}$, part (b) implies $\Psi_2(0)>\rho_1$, hence $S_2<\bar{S}_2$. Before creator 2 is activated, $\mathcal{S}_2(t)=S_2$, so the bound is immediate. Whenever creator 2 receives traffic, the allocation continues only while its marginal value is at least the currently active marginal value, and that value is never below $\rho_1$. Thus $\Psi_2(t)\geq \rho_1$ for all finite $t$. Since $\Psi_2(\cdot)$ is strictly decreasing, this implies $\mathcal{S}_2(t)\leq \bar{S}_2$ for all $t\geq 0$.

Finally, once the system enters the balanced phase, $\Psi_1(t)=\Psi_2(t)$. By Part (a), $\lim_{t\to\infty}\Psi_1(t)\to \rho_1$, so also $\Psi_2(t)\to \rho_1$. By uniqueness of the solution to $\Psi_2(s)=\rho_1$, we conclude that $\lim_{t\to\infty}\mathcal{S}_2(t)=\bar{S}_2=\sup_{t\geq 0}\mathcal{S}_2(t)$.

The comparative statics of $\bar{S}_{2,\mathrm{linear}}$ follow directly from the explicit formula that it is increasing in $\rho_2$, $p$, and $r$.

\subsection{Bass Growth}
Now we show the corresponding results under Bass growth. Recall that under the Bass model,
\begin{align*}
    \Phi_i(t)=\rho_i\left[1+\frac r\delta
    \left(p\rho_i+q\frac{\mathcal{S}_i(t)}{m_i}\right)\left(1-\frac{\mathcal{S}_i(t)}{m_i}\right)\right]\text{ for }i=1,2.  
\end{align*}
Also, by Lemma \ref{lem:2Crtr_bass_partial_Sit}, $\Phi_i(\cdot)$ is strictly decreasing on the deceleration branch $[\tilde{S}_i,m_i]$, where
\begin{align*}
    \tilde{S}_i=\frac{m_i(q-p\rho_i)}{2q}.
\end{align*}
First, we prove Proposition \ref{Prop:2Crtr_Long_base} Part (b) under Bass growth, that is, there exists a threshold $\rho_2^{\min}$ such that creator 2 will only receive traffic if $\rho_2> \rho_2^{\min}$. Before creator 2 is first activated, its follower base remains fixed at $S_2$. Hence, under the switching rule in Theorem \ref{Thm:Opt_policy_2Crtr_Bs}, increasing in $\rho_2$ raises creator 2's marginal value path and can only make the activation condition easier to satisfy. In particular, if creator 2 receives positive traffic at some capability level $\rho_2$, then it will also receive positive traffic at any larger capability level.

We next derive an explicit upper bound. A sufficient condition for creator 2 to be activated is $\phi_2>\rho_1$. Indeed, if $\phi_2\geq \phi_1$, then creator 2 is served immediately. If instead $\phi_1>\phi_2>\rho_1$, then creator 1 is initially served, and its marginal value decreases continuously toward $\rho_1$. If $S_2\geq \tilde{S}_2$, Theorem \ref{Thm:Opt_policy_2Crtr_Bs} implies that creator 2 is activated once $\Phi_1$ reaches $\phi_2$, which occurs in finite time because $\rho_1<\phi_2<\phi_1$. If $S_2<\tilde{S}_2$, activation occurs even earlier, since the switching condition is $\Phi_1-\phi_2\leq \eta_2$ with $\eta_2\geq 0$. Thus $\phi_2>\rho_1$ is sufficient for creator 2 to receive positive traffic.

Now write
\begin{align*}
    \phi_2=
    \rho_2\left[1+\frac r\delta\left(p\rho_2+q\frac{S_2}{m_2}\right)\left(1-\frac{S_2}{m_2}\right)\right]=
    a_2\rho_2^2+b_2\rho_2,
\end{align*}
where
\begin{align*}
    a_2:=\frac{pr}{\delta}\left(1-\frac{S_2}{m_2}\right),
    \quad
    b_2:=1+\frac{rq}{\delta}\frac{S_2}{m_2}\left(1-\frac{S_2}{m_2}\right).
\end{align*}
Hence the sufficient condition $\phi_2>\rho_1$ is equivalent to $a_2\rho_2^2+b_2\rho_2-\rho_1>0$. The corresponding positive root provides an explicit upper bound,
\begin{align*}
    \rho_2^{\min}\leq \rho_{2,\mathrm{bass}}^{\,\min}
    =
    \frac{-b_2+\sqrt{b_2^2+4a_2\rho_1}}{2a_2}.
\end{align*}
Equivalently,
\begin{align*}
    \rho_2^{\min}\leq \rho_{2,\mathrm{bass}}^{\,\min}
    =
    \frac{-1+\sqrt{\left[1+\frac{rqS_2}{m_2\delta}\left(1-\frac{S_2}{m_2}\right)\right]^2+4\rho_1\cdot\frac{pr}{\delta}\left(1-\frac{S_2}{m_2}\right)}}{2\cdot\frac{pr}{\delta}\left(1-\frac{S_2}{m_2}\right)}
    -\frac{qS_2}{2pm_2}.
\end{align*}
To establish the monotonicity of this upper bound, define $F(\rho_2;p,r,q):=
    a_2\rho_2^2+b_2\rho_2-\rho_1$. At $\rho_2=\rho_{2,\mathrm{bass}}^{\,\min}$,
\begin{align*}
    \frac{\partial F}{\partial \rho_2}=2a_2\rho_2+b_2>0.
\end{align*}
Moreover,
\begin{align*}
    \frac{\partial F}{\partial p}
    =
    \frac r\delta\left(1-\frac{S_2}{m_2}\right)\rho_2^2>0,
\end{align*}
\begin{align*}
    \frac{\partial F}{\partial r}
    =
    \frac p\delta\cdot \left(1-\frac{S_2}{m_2}\right)\rho_2^2
    +
    \frac{qS_2}{\delta m_2}\left(1-\frac{S_2}{m_2}\right)\rho_2>0,
\end{align*}
and
\begin{align*}
    \frac{\partial F}{\partial q}
    =
    \frac{rS_2}{m_2\delta}\left(1-\frac{S_2}{m_2}\right)\rho_2>0.
\end{align*}
Hence, by differentiation,
\begin{align*}
    \frac{\partial \rho_{2,\mathrm{bass}}^{\,\min}}{\partial p}<0,\quad
    \frac{\partial \rho_{2,\mathrm{bass}}^{\,\min}}{\partial r}<0,\quad
    \frac{\partial \rho_{2,\mathrm{bass}}^{\,\min}}{\partial q}<0.
\end{align*}
Then, we proceed to prove Proposition \ref{Prop:2Crtr_Long_base} Part (a) under Bass growth. If creator 2 never receives traffic, then $\mathcal{A}_1^*(t)\equiv A$, so $\mathcal{S}_1(t)\uparrow m_1$. Now suppose creator 2 is activated. By Theorem \ref{Thm:Opt_policy_2Crtr_Bs}, after at most one conditional reversal the system enters the balanced phase on the curve $\mathcal{S}_1=g_1(\mathcal{S}_2)$, and thereafter the optimal allocation is $(\mathcal{A}_1^*,\mathcal{A}_2^*)=(\bar{h}(\mathcal{S}_2),A-\bar{h}(\mathcal{S}_2))$.

On this balanced phase both creators are on their deceleration branches. Hence, by Lemma \ref{lem:2Crtr_bass_partial_Sit}, $\partial\Phi_i/\partial\mathcal{S}_i<0$ for $i=1,2$. By the closed-form expression of $\bar{h}(\cdot)$ in Theorem \ref{Thm:Opt_policy_2Crtr_Bs}, it follows that $0<\bar{h}(\mathcal{S}_2)<A$ whenever $\mathcal{S}_1<m_1$ and $\mathcal{S}_2<m_2$. Therefore
\begin{align*}
    \dot{\mathcal{S}}_1(t)
    =
    \bar{h}(\mathcal{S}_2(t))\cdot
    \mathcal{F}_B(\mathcal{S}_1(t);\rho_1,m_1)>0
\end{align*}
as long as $\mathcal{S}_1(t)<m_1$. Since $\mathcal{S}_1(t)$ is increasing and bounded above by $m_1$, it follows that $\lim_{t\to\infty}\mathcal{S}_1(t)=m_1$.

Finally, we prove Part (c) of the Proposition. Assume $\rho_2>\rho_2^{\min}$, so creator 2 is activated. By Theorem \ref{Thm:Opt_policy_2Crtr_Bs}, creator 2 enters the long-run balanced regime only on its deceleration branch. Let $\bar{S}_2$ denote the unique point on $[\tilde{S}_2,m_2)$ such that $\Phi_2(\bar{S}_2)=\rho_1$. Since $\Phi_2(\cdot)$ is strictly decreasing on $[\tilde{S}_2,m_2]$, this solution is unique. Solving
\begin{align*}
    \rho_2\left[1+\frac r\delta\left(p\rho_2+q\frac{\bar{S}_2}{m_2}\right)\cdot \left(1-\frac{\bar{S}_2}{m_2}\right)\right]=\rho_1
\end{align*}
gives
\begin{align*}
    \bar{S}_2=\bar{S}_{2,\mathrm{bass}}
    =
    \frac{m_2}{2q}\left[-(p\rho_2-q)+\sqrt{(p\rho_2+q)^2-\frac{4q\delta}{r}\left(\frac{\rho_1}{\rho_2}-1\right)}\right].
\end{align*}
In particular, $\bar{S}_2<m_2$.

We next show that $\mathcal{S}_2(t)\leq \bar{S}_2$ for all $t\geq 0$. If $S_2<\tilde{S}_2$, then $S_2<\bar{S}_2$. If $S_2\geq \tilde{S}_2$ and creator 2 is ever activated, then necessarily $\phi_2>\rho_1$, because otherwise creator 1's marginal value, which decreases only down to $\rho_1$, could never reach creator 2's initial marginal value. Since $\Phi_2(\cdot)$ is strictly decreasing on $[\tilde{S}_2,m_2]$ and $\Phi_2(\bar{S}_2)=\rho_1$, this implies $S_2<\bar{S}_2$.

During any initial phase in which creator 2 receives all traffic, the allocation stops once $\Phi_2$ reaches the currently active marginal value, and that value is always strictly larger than $\rho_1$. Hence creator 2 cannot reach a follower number exceeding $\bar{S}_2$. Once the balanced phase starts, we have $\Phi_2(t)=\Phi_1(t)>\rho_1$ for every finite $t$, because $\mathcal{S}_1(t)<m_1$. Since $\Phi_2(\cdot)$ is decreasing on the deceleration branch, this implies $\mathcal{S}_2(t)\leq \bar{S}_2$ for all $t\geq 0$.

Finally, by Part (a), $\Phi_1(t)\to \rho_1$, and on the balanced phase $\Phi_1(t)=\Phi_2(t)$. Hence $\Phi_2(t)\to \rho_1$, and by uniqueness of the deceleration branch solution to $\Phi_2(s)=\rho_1$, $\lim_{t\to\infty}\mathcal{S}_2(t)=\bar{S}_2=\sup_{t\geq 0}\mathcal{S}_2(t)$.

To establish the comparative statics of $\bar{S}_{2,\mathrm{bass}}$, define
\begin{align*}
    G(s;\rho_2,p,r,q)
    :=
    \rho_2\left[1+\frac r\delta\left(p\rho_2+q\frac s{m_2}\right)\cdot\left(1-\frac s{m_2}\right)\right]-\rho_1.
\end{align*}
At $s=\bar{S}_2$, we have $G(\bar{S}_2;\rho_2,p,r,q)=0$. Since $\bar{S}_2$ lies on the deceleration branch, Lemma \ref{lem:2Crtr_bass_partial_Sit} gives
\begin{align*}
    \frac{\partial G}{\partial s}(\bar{S}_2;\rho_2,p,r,q)<0.
\end{align*}
Moreover, we have
\begin{align*}
    \frac{\partial G}{\partial \rho_2}
    =
    1+\frac r\delta\left(2p\rho_2+q\frac{\bar{S}_2}{m_2}\right)\left(1-\frac{\bar{S}_2}{m_2}\right)>0,
\end{align*}
\begin{align*}
    \frac{\partial G}{\partial p}
    =
    \frac{r\rho_2^2}{\delta}\left(1-\frac{\bar{S}_2}{m_2}\right)>0,
\end{align*}
\begin{align*}
    \frac{\partial G}{\partial r}
    =
    \frac{\rho_2}{\delta} \left(p\rho_2+q\frac{\bar{S}_2}{m_2}\right)\left(1-\frac{\bar{S}_2}{m_2}\right)>0,
\end{align*}
and
\begin{align*}
    \frac{\partial G}{\partial q}
    =
    \frac{r\rho_2\bar{S}_2}{m_2\delta}\left(1-\frac{\bar{S}_2}{m_2}\right)>0.
\end{align*}
Thus,
\begin{align*}
    \frac{\partial \bar{S}_{2,\mathrm{bass}}}{\partial \rho_2}>0,\quad
    \frac{\partial \bar{S}_{2,\mathrm{bass}}}{\partial p}>0,\quad
    \frac{\partial \bar{S}_{2,\mathrm{bass}}}{\partial r}>0,\quad
    \frac{\partial \bar{S}_{2,\mathrm{bass}}}{\partial q}>0.
\end{align*}
This completes the proof. \qed

\section{Proof of Theorem \ref{Thm:Opt_policy_NCrtr_Bs}}
We prove Theorem \ref{Thm:Opt_policy_NCrtr_Bs} in a similar way as the proof of Theorem \ref{Thm:Opt_policy_2Crtr_Bs}. We first derive a creator-specific switching coefficient from the Hamiltonian, then show that on any active set the platform should equalize marginal values, and finally verify recursively that the proposed switching rule maximizes the Hamiltonian at every stage.

As in the proof of Theorem \ref{Thm:Opt_policy_2Crtr_Bs}, we write
\begin{align*}
    \mathcal{F}_{B,i}(\mathcal{S}_i)=\left(p\rho_i+q\cdot\frac{\mathcal{S}_i}{m_i}\right)\left(1-\frac{\mathcal{S}_i}{m_i}\right),\quad \tilde{S}_i=\frac{m_i(q-p\rho_i)}{2q},
\end{align*}
and recall that
\begin{align*}
    \Phi_i(t)=\rho_i\left[1+\frac{r}{\delta}\mathcal{F}_{B,i}(\mathcal{S}_i(t))\right],\quad \phi_i=\Phi_i(S_i).
\end{align*}
Since a creator receives no traffic before being activated, every inactive creator $j$ stays at its initial follower base $S_j$ up to its activation time. Hence its marginal value remains equal to the constant $\phi_j$ throughout the inactive period. This allows us to define the activation threshold $\eta_j$ from the initial state $S_j$.

\subsection{Construction of the Proposed Policy}
We first define the full-traffic flow of creator $j$. For any $x\in[0,m_j]$, let $\Xi_j(t;x)$ be the unique solution to
\begin{align*}
    \frac{d}{dt}\Xi_j(t;x)=A\cdot\mathcal{F}_{B,j}(\Xi_j(t;x)),\;\Xi_j(0;x)=x.
\end{align*}
Next, let $M\subseteq\{1,\dots,N\}$ be a nonempty set, and suppose that at some time its member satisfy
\begin{align*}
    \Phi_i(s_i)=\Phi_{i'}(s_{i'}) \text{ for all }i,i'\in M,
\end{align*}
with $s_i\geq\tilde{S}_i$ for all $i\in M$. Define
\begin{align*}
    \beta_i(s_i):=-\rho_i\mathcal{F}_{B,i}'(s_i)\mathcal{F}_{B,i}(s_i)\text{ for }i\in M.
\end{align*}
Since $s_i\geq\tilde{S}_i$, we have $\mathcal{F}_{B,i}'(s_i)\leq 0$, so $\beta_i(s_i)>0$. Define the balanced allocation on $M$ by
\begin{align*}
    a_i^M(s_M):=A\cdot\frac{\beta_i(s_i)^{-1}}{\sum_{\ell\in M}\beta_{\ell}(s_{\ell})^{-1}}\text{ for }i\in M.
\end{align*}
Of which, $s_M$ is the vector of follower states of creators in the active set $M$. The next lemma shows that this allocation preserves equality of marginal values within the active set.

\begin{lemma}\label{lem:NCrtr_Bs_balanced_alloc}
Let $M\subseteq\{1,\dots,N\}$ be nonempty, and suppose that at time $t_0$ the state
$s_M=(s_i)_{i\in M}$ satisfies
\begin{align*}
\Phi_i(s_i)=\Phi_{i'}(s_{i'})\text{ for all }i,i'\in M,
\quad
s_i\ge \tilde S_i\text{ for all }i\in M.
\end{align*}
If traffic is allocated according to $a^M(\cdot)$, then the resulting trajectory
preserves the equality
\begin{align*}
\Phi_i(t)=\Phi_{i'}(t)\text{ for all }i,i'\in M\text{ for }t\geq t_0.
\end{align*}
Moreover, $a^M(\cdot)$ is admissible.
\end{lemma}
\textit{Proof. }Admissibility is immediate from $\beta_i(s_i)>0$, we have $a_i^M(s_M)\geq 0$ and
\begin{align*}
    \sum_{i\in M}a_i^M(s_M)=A.
\end{align*}
Now, under the state dynamics
\begin{align*}
    \dot\Phi_i(t)=\frac{r}{\delta}\rho_i\mathcal F'_{B,i}(\mathcal S_i(t))\dot{\mathcal S}_i(t)=\frac{r}{\delta}\rho_i\mathcal F'_{B,i}(\mathcal S_i(t))a_i^M(\mathcal S_M(t))\mathcal F_{B,i}(\mathcal S_i(t)).
\end{align*}
By the definition of $a_i^M$ and $\beta_i$,
\begin{align*}
    \dot\Phi_i(t)=-\frac{rA/\delta}{\sum_{\ell\in M}\beta_\ell(\mathcal S_\ell(t))^{-1}},
\end{align*}
which is identical for all $i\in M$. Hence all $\Phi_i(t)$ evolve with
the same derivative, so equality is preserved for all $t\geq t_0$.\qed

We now define the activation threshold $\eta_j$. If $S_j\geq \tilde S_j$, creator
$j$ is already on the decreasing branch, and we simply set $\eta_j=0$. Suppose now that $S_j<\tilde S_j$. Along the full-traffic path $t\mapsto \Xi_j(t;S_j)$, the marginal value $\Phi_j(\Xi_j(t;S_j))$ first increases, reaches its maximum at $\tilde{S}_j$, and then decreases. For every level $c\in[\phi_j,\Phi_j(\tilde S_j)]$, define
\begin{align*}
    v_j(c):=\inf\{t\ge 0:\Xi_j(t;S_j)\ge \tilde S_j,\ \Phi_j(\Xi_j(t;S_j))=c\},
\end{align*}
that is, the first time on the decreasing branch at which the marginal value returns to level $c$. Define
\begin{align*}
    \mathcal{E}_j(c):=\delta\int_0^{v_j(c)} e^{-\delta t}\Big[c-\Phi_j(\Xi_j(t;S_j))\Big]\,dt.
\end{align*}
The next lemma gives the existence and uniqueness of this threshold.
\begin{lemma}\label{lem:NCrtr_Bs_gamma_exist}
    Suppose $S_j<\tilde{S}_j$. Then $\mathcal{E}_j(c)$ is continuous and strictly increasing on $[\phi_j,\Phi_j(\tilde{S}_j)]$, with $\mathcal{E}_j(\phi_j)<0$, $\mathcal{E}_j(\Phi_j(\tilde{S}_j))>0$. Hence, there exists a unique level $\bar{c}_j\in(\phi_j,\Phi_j(\tilde{S}_j))$ such that $\mathcal{E}_j(\bar{c}_j)=0$. Setting
    \begin{align*}
        \eta_j:=\bar{c}_j-\phi_j
    \end{align*}
    defines a unique threshold $\eta_j>0$.
\end{lemma}
\textit{Proof. }Continuity of $\mathcal{E}_j$ follows from continuity of the flow $\Xi_j(\cdot;S_j)$ and of the return time $v_j(c)$. To prove strict monotonicity, differentiate $\mathcal{E}_j(c)$
with respect to $c$. Since the boundary term vanishes at $t=v_j(c)$,
\begin{align*}
    \mathcal{E}_j'(c)=\delta\int_0^{v_j(c)} e^{-\delta t}\,dt=1-e^{-\delta v_j(c)}>0.
\end{align*}
Hence, $\mathcal{E}_j$ is strictly increasing. At $c=\phi_j$, the full-traffic path starts at level $\phi_j$, then rises above $\phi_j$, and only later returns to $\phi_j$ on the decreasing branch. Therefore $\phi_j-\Phi_j(\Xi_j(t;S_j))<0$ for all $t\in(0,v_j(\phi_j))$, which implies $\mathcal{E}_j(\phi_j)<0$. At $c=\Phi_j(\tilde S_j)$, we have $v_j(c)=\inf\{t:\Xi_j(t;S_j)=\tilde S_j\}>0$, and $\Phi_j(\tilde S_j)-\Phi_j(\Xi_j(t;S_j))>0$ for all $t\in(0,v_j(c))$, so $\mathcal{E}_j(\Phi_j(\tilde S_j))>0$. By continuity and strict monotonicity, there exists a unique
$\bar c_j\in(\phi_j,\Phi_j(\tilde S_j))$ such that
$\mathcal{E}_j(\bar c_j)=0$. Equivalently, there exists a unique $\eta_j>0$ such that
\begin{align*}
    \delta\int_0^{v_j(\phi_j+\eta_j)} e^{-\delta t}
\Big[\phi_j+\eta_j-\Phi_j(\Xi_j(t;S_j))\Big]\,dt=0.
\end{align*}
This proves the claim.\qed

We now construct the policy recursively. Let $M(0)$ be the initial active set given in
the theorem. Without loss of generality, we suppose $M(0)=\{k\}$; the case of a larger initial active set is identical, starting directly with the balanced allocation from Lemma \ref{lem:NCrtr_Bs_balanced_alloc}. At stage $\ell$, the active set is $M_\ell$, and by construction all creators in $M_\ell$ have the same marginal value, which we denote by $\Phi_{M_\ell}(t)$.

If $|M_\ell|=1$, all traffic is allocated to the unique active creator. If
$|M_\ell|\ge 2$, traffic is allocated on $M_\ell$ according to
Lemma \ref{lem:NCrtr_Bs_balanced_alloc} so as to preserve equality of marginal values
within $M_\ell$. This continues until the first time $\tau_\ell$ at which some inactive
creator $j\notin M_\ell$ satisfies $\Phi_{M_\ell}(t)-\phi_j=\eta_j$. Choose one such creator and denote it by $j_{\ell}$.

If $S_{j_\ell}\geq \tilde{S}_{j_\ell}$, then $\eta_{j_\ell}=0$, so
$\Phi_{M_\ell}(\tau_\ell)=\phi_{j_\ell}$, and creator $j_\ell$ is added immediately to
the active set, $M_{\ell+1}=M_{\ell}\cup\{j_{\ell}\}$, $t_{\ell+1}:=\tau_{\ell}$. From time $t_{\ell+1}$ onwards, traffic is balanced on $M_{\ell+1}$ according to Lemma \ref{lem:NCrtr_Bs_balanced_alloc}.

If instead $S_{j_{\ell}}<\tilde{S}_{j_{\ell}}$, then $\eta_{j_{\ell}}>0$. In that case, starting from time $\tau_{\ell}$, all traffic is allocated to creator $j_{\ell}$ until the first time
\begin{align*}
    \sigma_\ell:=\inf\Big\{t\ge \tau_\ell:\mathcal S_{j_\ell}(t)\ge \tilde S_{j_\ell},\;\Phi_{j_\ell}(t)=\Phi_{M_\ell}(\tau_\ell)\Big\}.
\end{align*}
Then creator $j_{\ell}$ is added to the active set, $M_{\ell+1}=M_{\ell}\cup\{j_{\ell}\}$, $t_{\ell+1}:=\sigma_{\ell}$. And from time $t_{\ell+1}$ onwards traffic is balanced on $M_{\ell+1}$ according to Lemma \ref{lem:NCrtr_Bs_balanced_alloc}. This defines the proposed policy recursively.

\subsection{Verification of Optimality}
We now verify that the above policy satisfies the sufficient conditions of the
Pontryagin Maximum Principle. Define for each creator $i$,
\begin{align*}
    Y_i(t):=\rho_i+\lambda_i(t)\mathcal F_{B,i}(\mathcal S_i(t)).
\end{align*}
Then, the Hamiltonian can be written as
\begin{align*}
    \mathcal{H}=\sum_{i=1}^N Y_i(t)\mathcal A_i(t)
+\sum_{i=1}^N \rho_i r\mathcal S_i(t).
\end{align*}
The next lemma generalizes the switching function identity used in the 2-creator proof (Lemma \ref{lem:Et_dynamics}).
\begin{lemma}\label{lem:NCrtr_Bs_Y_dynamics}
For every $i=1,\dots,N$,
\begin{align*}
\dot Y_i(t)=\delta Y_i(t)-\delta \Phi_i(t).
\end{align*}
Consequently, for any pair $i,j$, $E_{i,j}(t):=Y_i(t)-Y_j(t)$ satisfies
\begin{align*}
\dot E_{i,j}(t)=\delta E_{i,j}(t)-\delta\big(\Phi_i(t)-\Phi_j(t)\big).
\end{align*}
Suppose that
$\lim_{t\to\infty}e^{-\delta t}\lambda_i(t)=0$, we further have
\begin{align*}
E_{i,j}(t)=\delta\int_t^\infty e^{\delta(t-s)}
\big(\Phi_i(s)-\Phi_j(s)\big)\,ds.
\end{align*}
\end{lemma}
\textit{Proof. }Differentiating $Y_i(t)$ gives
\begin{align*}
\dot Y_i(t)=\dot\lambda_i(t)\mathcal F_{B,i}(\mathcal S_i(t))
+\lambda_i(t)\mathcal F'_{B,i}(\mathcal S_i(t))\dot{\mathcal S}_i(t).
\end{align*}
Using the costate equation and
$\dot{\mathcal S}_i(t)=\mathcal A_i(t)\mathcal F_{B,i}(\mathcal S_i(t))$, the terms
involving $\mathcal A_i(t)\mathcal F'_{B,i}$ cancel, and we obtain
\begin{align*}
\dot Y_i(t)=\delta\lambda_i(t)\mathcal F_{B,i}(\mathcal S_i(t))
-\rho_i r\,\mathcal F_{B,i}(\mathcal S_i(t))=
\delta Y_i(t)-\delta\Phi_i(t).
\end{align*}
Subtracting the equations for $i$ and $j$ yields the equation for $E_{i,j}$. Multiplying
by $e^{-\delta t}$ and integrating forward to $\infty$ gives the integral
representation.\qed

The next lemma is the direct $N$-creator counterpart of the positivity lemma in the 2-creator proof.
\begin{lemma}\label{lem:NCrtr_Bs_lambda_positive}
    If the transversality condition $\lim_{t\to\infty}e^{-\delta t}\lambda_i(t)=0$ holds for all $i$, then $\lambda_i(t)>0$ for every $i$ and every $t>0$.
\end{lemma}
\textit{Proof. }Fix $i$, and write
\begin{align*}
a_i(t):=\delta-\mathcal A_i(t)\mathcal F'_{B,i}(\mathcal S_i(t)),
\quad
M_i(t):=\exp\left(-\int_0^t a_i(u)\,du\right).
\end{align*}
Then
\begin{align*}
\frac{d}{dt}\big(\lambda_i(t)M_i(t)\big)=-\rho_i r\,M_i(t),
\end{align*}
so for any $T>t$,
\begin{align*}
\lambda_i(t)=\frac{\lambda_i(T)M_i(T)}{M_i(t)}
+\frac{\rho_i r}{M_i(t)}\int_t^T M_i(s)\,ds.
\end{align*}
Using
\begin{align*}
M_i(T)=e^{-\delta T}\frac{\mathcal F_{B,i}(\mathcal S_i(T))}{\mathcal F_{B,i}(\mathcal S_i(0))}
\end{align*}
and then letting $T\to\infty$, the transversality condition implies
\begin{align*}
\lambda_i(t)=\frac{\rho_i r}{M_i(t)}\int_t^\infty M_i(s)\,ds>0.
\end{align*}
Thus $\lambda_i(t)>0$.\qed

The next lemma states the sufficient condition as in the 2-creator proof.
\begin{lemma}\label{lem:NCrtr_Bs_PMP_sufficient}
    For any admissible policy $\mathcal{A}_i(t)$ and corresponding state trajectories $\mathcal{S}_i(t)$, $i=1,\dots,N$ such that 1) the transversality condition $\lim_{t\to\infty} e^{-\delta t}\lambda_i(t)=0$; 2) the Hamiltonian $\mathcal{H}$ is concave in $\mathcal{S}$ at each $t$; 3) $\mathcal{A}_1^*(t),\dots,\mathcal{A}_N^*(t)$ are admissible, and the proposed control allocates positive traffic only to creators with maximal $Y(t)$.
\end{lemma}
\textit{Proof.} Same as the proof of Lemma \ref{lem:PMP_suff_conds}, all conditions extend to $N$-creator case.\qed

We next verify concavity and some structural properties of the proposed policy.
\begin{lemma}\label{lem:NCrtr_Bs_concavity}
The Hamiltonian is concave in $(\mathcal S_1,\dots,\mathcal S_N)$ at each $t$.
\end{lemma}
\textit{Proof. }Since the state variables enter the Hamiltonian separately across creators,
\begin{align*}
\frac{\partial^2\mathcal H}{\partial \mathcal S_i^2}=\lambda_i(t)\mathcal A_i(t)\mathcal F''_{B,i}(\mathcal S_i(t)).
\end{align*}
Of which, $\mathcal F''_{B,i}(s)=-\frac{2q}{m_i^2}<0$, $\mathcal{A}_i(t)\geq 0$, and by Lemma \ref{lem:NCrtr_Bs_lambda_positive}, $\lambda_i(t)>0$. Hence $\frac{\partial^2\mathcal H}{\partial \mathcal S_i^2}\le 0$ for all $i$, which proves concavity.\qed

\begin{lemma}\label{lem:NCrtr_Bs_Y_equal_in_active_set}
Suppose $M$ is an active set and the balanced allocation from
Lemma \ref{lem:NCrtr_Bs_balanced_alloc} is used on $M$. If at some time $t_0$ the
creators in $M$ satisfy
\begin{align*}
\Phi_i(t_0)=\Phi_{i'}(t_0),
Y_i(t_0)=Y_{i'}(t_0)
\text{ for all }i,i'\in M,
\end{align*}
then these equalities continue to hold for all $t\geq t_0$.
\end{lemma}
\textit{Proof. }By Lemma \ref{lem:NCrtr_Bs_balanced_alloc}, equality of the $\Phi_i$'s is preserved.
Once $\Phi_i(t)=\Phi_{i'}(t)$, Lemma \ref{lem:NCrtr_Bs_Y_dynamics} gives
\begin{align*}
\frac{d}{dt}\big(Y_i(t)-Y_{i'}(t)\big)=\delta\big(Y_i(t)-Y_{i'}(t)\big).
\end{align*}
Since the difference is zero at $t_0$, it remains zero for all $t\geq t_0$.\qed

Hence, whenever the system is on an active set $M$, we may write the common switching
coefficient and common marginal value as $Y_M(t):=Y_i(t)$, $\Phi_M(t):=\Phi_i(t)$, for any $i\in M$.

We now record the boundary properties implied by the construction of the thresholds.
\begin{lemma}\label{lem:NCrtr_Bs_boundary_properties}
Fix an active set $M$ and an inactive creator $j\notin M$.
\begin{itemize}
    \item[(a)] If $S_j\ge \tilde S_j$ and $\Phi_M(t_0)=\phi_j$, then $Y_M(t_0)=Y_j(t_0)$.
    \item[(b)] If $S_j<\tilde S_j$ and $\Phi_M(t_0)-\phi_j=\eta_j$, then $Y_M(t_0)=Y_j(t_0)$.
\end{itemize}
\end{lemma}
\textit{Proof. } In case (a), creator $j$ is already on the decreasing branch and $\eta_j=0$.
By construction, if the system is at a state with $\Phi_M(t_0)=\phi_j$, then creator
$j$ is admitted immediately and thereafter the enlarged active set is balanced, so the
future gap $\Phi_M-\Phi_j$ is identically zero. By
Lemma \ref{lem:NCrtr_Bs_Y_dynamics},
\begin{align*}
    Y_M(t_0)-Y_j(t_0)=\delta\int_{t_0}^\infty e^{\delta(t_0-s)}\big(\Phi_M(s)-\Phi_j(s)\big)\,ds=0.
\end{align*}
In case (b), let $c:=\phi_j+\eta_j=\Phi_M(t_0)$. By the definition of
$\eta_j$,
\begin{align*}
    \delta\int_0^{v_j(c)} e^{-\delta u}\big[c-\Phi_j(\Xi_j(u;S_j))\big]\,du=0.
\end{align*}
If all traffic is allocated to creator $j$ from time $t_0$ until the return time
$t_0+v_j(c)$, while the old active set remains the same, then $\Phi_M(s)\equiv c$ on that
interval and the future gap is zero afterwards. Hence
Lemma \ref{lem:NCrtr_Bs_Y_dynamics} yields
\begin{align*}
    Y_M(t_0)-Y_j(t_0)=\delta\int_{t_0}^{t_0+v_j(c)} e^{\delta(t_0-s)}\big(\Phi_M(s)-\Phi_j(s)\big)\,ds=0.
\end{align*}
This proves both parts.\qed

\begin{lemma}\label{lem:NCrtr_Bs_reversal_sign}
In the setting of Lemma \ref{lem:NCrtr_Bs_boundary_properties}(b), if all traffic is
allocated to creator $j$ immediately after time $t_0$, then
\begin{align*}
Y_M(t)-Y_j(t)<0 \text{ for all }t\in(t_0,t_0+v_j(c)),
\end{align*}
where $c=\phi_j+\eta_j$.
\end{lemma}
\textit{Proof. }Write $D(t):=\Phi_M(t)-\Phi_j(t)$. During the reversal phase, the old active set is unchanged, $\Phi_M(t)\equiv c$, while creator $j$ evolves under the full-traffic flow. Thus $D(t)=c-\Phi_j(\Xi_j(t-t_0;S_j))$. By construction of $v_j(c)$, the function $D(t)$ is positive at first, then negative, and finally returns to zero at $t=t_0+v_j(c)$. Define
\begin{align*}
F(t):=\int_t^{t_0+v_j(c)} e^{-\delta s}D(s)\,ds.
\end{align*}
Since $\Gamma_j(c)=0$, we have $F(t_0)=0$. Moreover, $F'(t)=-e^{-\delta t}D(t)$. Hence $F$ decreases strictly from $0$ while $D>0$, and then increases only after $D<0$. Since it returns to $0$ only at the terminal time $t_0+v_j(c)$, we must have $F(t)<0$ for all $t\in(t_0,t_0+v_j(c))$. By Lemma \ref{lem:NCrtr_Bs_Y_dynamics},
\begin{align*}
    Y_M(t)-Y_j(t)=\delta e^{\delta t}F(t)<0,
\end{align*}
which proves the claim.\qed

We now complete the verification by induction over stages. At time $0$, the active set
is $M_0=\{k\}$, and by assumption
\begin{align*}
    \phi_k-\phi_j>\eta_j\text{ for all }j\neq k.
\end{align*}
Let $\tau_j$ be the first time at which $\Phi_k(t)-\phi_j=\eta_j$. For every $t<\tau_j$, splitting the integral representation of Lemma \ref{lem:NCrtr_Bs_Y_dynamics} at $\tau_j$ and using
Lemma \ref{lem:NCrtr_Bs_boundary_properties} gives
\begin{align*}
    Y_k(t)-Y_j(t)=\delta\int_t^{\tau_j} e^{\delta(t-s)}(\Phi_k(s)-\phi_j)\,ds>0.
\end{align*}
Hence creator $k$ has the largest switching coefficient and allocating all traffic to
creator $k$ is optimal up to the first switching time $\tau_0:=\min_{j\neq k}\tau_j$.

Suppose now that at the beginning of stage $\ell$, the active set is $M_\ell$, all
creators in $M_\ell$ share the same marginal value and the same switching coefficient,
and every inactive creator $j\notin M_\ell$ satisfies
\begin{align*}
    \Phi_{M_\ell}(t_\ell)-\phi_j>\eta_j.
\end{align*}
If $|M_\ell|=1$, all traffic is allocated to the unique active creator. If $|M_\ell|\geq 2$, traffic is balanced on $M_\ell$ according to Lemma \ref{lem:NCrtr_Bs_balanced_alloc}. In either case, let $j_\ell$ be the creator that first reaches its threshold, so
\begin{align*}
    \Phi_{M_\ell}(\tau_\ell)-\phi_{j_\ell}=\eta_{j_\ell}.
\end{align*}
For every inactive $j\notin M_{\ell}$, the same argument as above yields $Y_{M_\ell}(t)-Y_j(t)>0$ for $t\in[t_\ell,\tau_\ell)$, so every active creator dominates every inactive creator before the next switching time. Thus the theorem's prescription is pointwise Hamiltonian-maximizing on $[t_{\ell},\tau_{\ell})$.

If $S_{j_\ell}\geq \tilde S_{j_\ell}$, then $\eta_{j_\ell}=0$, and
Lemma \ref{lem:NCrtr_Bs_boundary_properties}(a) gives $Y_{M_\ell}(\tau_\ell)=Y_{j_\ell}(\tau_\ell)$. Hence, creator $j_\ell$ can be admitted immediately to the active set. The balanced
allocation on $M_{\ell+1}=M_\ell\cup\{j_\ell\}$ is admissible by Lemma \ref{lem:NCrtr_Bs_balanced_alloc}, and equality of both $\Phi_i$ and $Y_i$ on
$M_{\ell+1}$ is preserved by Lemmas \ref{lem:NCrtr_Bs_balanced_alloc} and
\ref{lem:NCrtr_Bs_Y_equal_in_active_set}.

If $S_{j_\ell}<\tilde S_{j_\ell}$, then Lemma \ref{lem:NCrtr_Bs_boundary_properties}(b) gives $Y_{M_\ell}(\tau_\ell)=Y_{j_\ell}(\tau_\ell)$, while Lemma \ref{lem:NCrtr_Bs_reversal_sign} shows that, once the conditional reversal starts, $Y_{M_\ell}(t)-Y_{j_\ell}(t)<0$ for $t\in(\tau_{\ell},\sigma_{\ell})$.

Thus, creator $j_\ell$ strictly dominates the old active set throughout the reversal
interval. Moreover, any other inactive creator $h\notin M_\ell\cup\{j_\ell\}$ remains
unchanged during the reversal phase, so $\phi_h$ remains unchanged and $\Phi_{M_\ell}(t)-\phi_h>\eta_h$. Hence, $Y_{M_\ell}(t)-Y_h(t)>0$ for $t\in(\tau_{\ell},\sigma_{\ell})$, and therefore
\begin{align*}
    Y_{j_\ell}(t)-Y_h(t)=\big(Y_{j_\ell}(t)-Y_{M_\ell}(t)\big)+\big(Y_{M_\ell}(t)-Y_h(t)\big)>0.
\end{align*}
So creator $j_\ell$ dominates every other inactive creator as well. Therefore allocating all traffic to $j_\ell$ on $(\tau_\ell,\sigma_\ell)$ is pointwise Hamiltonian-maximizing. At time $\sigma_{\ell}$, by construction, $\Phi_{j_\ell}(\sigma_\ell)=\Phi_{M_\ell}(\sigma_\ell)$, and the same identity used in Lemma \ref{lem:NCrtr_Bs_boundary_properties}(b) gives $Y_{j_\ell}(\sigma_\ell)=Y_{M_\ell}(\sigma_\ell)$. Hence, creator $j_{\ell}$ may be added to the active set, and balancing on $M_{\ell+1}=M_{\ell}\cup\{j_{\ell}\}$ is again admissible and preserves equality of both $\Phi_i$ and $Y_i$ on the enlarged set.

By induction, the recursively defined policy maximizes the Hamiltonian at every stage.
Since each stage enlarges the active set by one creator, the process terminates after at
most $N-|M_0|$ enlargements. If no inactive creator ever reaches its threshold at some
stage, then the current allocation continues indefinitely. Finally, Lemmas \ref{lem:NCrtr_Bs_lambda_positive},
\ref{lem:NCrtr_Bs_PMP_sufficient}, and
\ref{lem:NCrtr_Bs_concavity} verify the sufficient conditions of the Pontryagin Maximum
Principle. Hence the proposed policy is optimal. This completes the proof of
Theorem \ref{Thm:Opt_policy_NCrtr_Bs}.

\section{Proof of Proposition \ref{Prop:NCrtr_Bs_Long_base}}
We use Theorem \ref{Thm:Opt_policy_NCrtr_Bs} and extend the arguments of Proposition \ref{Prop:2Crtr_Long_base}.

We first show that creator 1 must belong to the active set after finitely many switching steps. Suppose otherwise, then creator 1 remains inactive forever, so its follower base stays fixed at $S_1$, and its activation threshold remains $\phi_1+\eta_1$, where $\phi_1+\eta_1\geq \phi_1>\rho_1$.

If creator 1 is never activated, then there is a final active set $M^\infty$ that excludes creator 1. Let $j^\infty:=\min M^\infty$, i.e., the creator in $M^\infty$ with the highest capability. On the final stage, the optimal policy maintains equality of marginal values within $M^\infty$, so the common active marginal value is $\Phi_{M^\infty}(t)=\Phi_j(t)$ for all $j\in M^{\infty}$.

Since every creator in the final active set receives positive traffic, the same argument as in Proposition \ref{Prop:2Crtr_Long_base}(a) implies that creator $j^\infty$ must asymptotically reach its full capacity, and therefore $\Phi_{M^\infty}(t)\to \rho_{j^\infty}<\rho_1$.

But then the switching condition for creator 1 in Theorem \ref{Thm:Opt_policy_NCrtr_Bs} must eventually be satisfied, because $\phi_1+\eta_1>\rho_1>\rho_{j^\infty}$, contradicting the assumption that creator 1 is never activated. Therefore creator 1 joins the active set after finitely many stages.

Once creator 1 enters the active set, it will never leave since the active set only enlarges. After finitely many further enlargements, the process reaches its final stage. If creator 1 is the only active creator on that final stage, then it receives all traffic forever and $\mathcal{S}_1(t)\uparrow m_1$. Otherwise creator 1 belongs to a balanced active set, and by Lemma \ref{lem:NCrtr_Bs_balanced_alloc} it receives a strictly positive traffic share forever after. Therefore $\dot{\mathcal{S}}_1(t)>0$ whenever $\mathcal{S}_1(t)<m_1$ on the final stage. Since $\mathcal{S}_1(t)$ is increasing and bounded above by $m_1$, it follows that $\lim_{t\to\infty}\mathcal{S}_1(t)=m_1$.

Now we proceed to prove Part (b) of Proposition \ref{Prop:NCrtr_Bs_Long_base}. Fix $i\neq 1$, before creator $i$ is first activated, its follower base remains fixed at $S_i$. Under the recursive switching rule of Theorem \ref{Thm:Opt_policy_NCrtr_Bs}, increasing $\rho_i$ raises creator $i$'s own marginal value path and can only make its activation condition easier to satisfy. Hence, if creator $i$ receives positive traffic at some capability level $\rho_i$, then it will also receive positive traffic at any larger capability level.

We next derive the explicit upper bound. By part (a), creator 1 eventually belongs to the active set, and from then on the common active marginal value decreases to $\rho_1$. Therefore a sufficient condition for creator $i$ to be activated is $\phi_i>\rho_1$, because then also $\phi_i+\eta_i>\rho_1$, so at some finite time the common active marginal value must cross creator $i$'s activation threshold.

Now write
\begin{align*}
    \phi_i
    =
    \rho_i\left[1+\frac r\delta\left(p\rho_i+q\frac{S_i}{m_i}\right)\left(1-\frac{S_i}{m_i}\right)\right]
    =
    a_i\rho_i^2+b_i\rho_i,
\end{align*}
where
\begin{align*}
    a_i:=\frac{pr}{\delta}\left(1-\frac{S_i}{m_i}\right),
    \quad
    b_i:=1+\frac{rqS_i}{m_i\delta}\left(1-\frac{S_i}{m_i}\right).
\end{align*}
Hence $\phi_i>\rho_1$ is equivalent to $a_i\rho_i^2+b_i\rho_i-\rho_1>0$. Thus,
\begin{align*}
    \rho_i^{\min}\leq \rho_{i,\mathrm{bass}}^{\,\min}
    =
    \frac{-b_i+\sqrt{b_i^2+4a_i\rho_1}}{2a_i}.
\end{align*}
Equivalently,
\begin{align*}
    \rho_i^{\min}\leq \rho_{i,\mathrm{bass}}^{\,\min}
    =
    \frac{-1+\sqrt{\left[1+\frac{rqS_i}{m_i\delta}\left(1-\frac{S_i}{m_i}\right)\right]^2+4\rho_1\cdot\frac{pr}{\delta}\left(1-\frac{S_i}{m_i}\right)}}{2\cdot\frac{pr}{\delta}\left(1-\frac{S_i}{m_i}\right)}
    -\frac{qS_i}{2pm_i}.
\end{align*}
Next, we prove Part (c) of Proposition \ref{Prop:NCrtr_Bs_Long_base}. Fix $i\neq 1$ with $\rho_i>\rho_i^{\min}$, so creator $i$ is activated at some finite time. Once activated, creator $i$ remains in the active set forever. Since the active set enlarges only finitely many times, there is a final stage at which creator 1 and creator $i$ both belong to the active set, and the optimal allocation maintains equality of marginal values within that set, $\Phi_j(t)=\Phi_{M}(t)$ for all active $j$.

By part (a), $\mathcal{S}_1(t)\to m_1$, hence $\Phi_M(t)=\Phi_1(t)\to \rho_1$. Therefore, $\Phi_i(t)\to\rho_1$ as $t\to\infty$. Now let $\bar{S}_i$ be the unique value on $[\tilde{S}_i,m_i)$ satisfying $\Phi_i(\bar{S}_i)=\rho_1$.

Exactly as in the two-creator case, this is the larger root of
\begin{align*}
    \rho_i\left[1+\frac r\delta\left(p\rho_i+q\frac{\bar{S}_i}{m_i}\right)\left(1-\frac{\bar{S}_i}{m_i}\right)\right]=\rho_1,
\end{align*}
and we solve for
\begin{align*}
    \bar{S}_i=
    \frac{m_i}{2q}\left[-(p\rho_i-q)+\sqrt{(p\rho_i+q)^2-\frac{4q\delta}{r}\left(\frac{\rho_1}{\rho_i}-1\right)}\right].
\end{align*}
In particular, $\bar{S}_i<m_i$. Similar as in the Proof of Proposition \ref{Prop:2Crtr_Long_base}, $\mathcal{S}_i(t)\leq\bar{S}_i$ for all $t\geq 0$ and $\lim_{t\to\infty}\mathcal{S}_i(t)=\bar{S}_i=\sup_{t\geq 0}\mathcal{S}_i(t)$. The monotonicity results in Part (d) follow directly from the same formulas as in the two-creator Bass case.\qed

\end{document}